\newif\ifsiam     
\newif\ifnummat   
\newif\ifcvs      
\newif\ifmoc      
\newif\ifanm      
\newif\ifjcam     
\newif\ifjcm      

\siamfalse
\nummatfalse
\cvsfalse
\mocfalse
\anmfalse
\jcamfalse
\jcmfalse
 
\siamtrue

\ifsiam
    \documentclass[final]{siamart190516}
    \newtheorem{remark}{Remark}[section]

    \newtheorem{assumption}{Assumption}[section]
\fi

\ifnummat
    \RequirePackage{fix-cm}
    \documentclass{svjour3}
    \journalname{Numerische Mathematik}
    \smartqed
    \DeclareMathAlphabet{\mathcal}{OMS}{cmsy}{m}{n}

\fi

\ifcvs
    \documentclass[twocolumn,fleqn,runningheads,final]{svjour3}
    \journalname{Computing and Visualization in Science}
    \usepackage{mathptmx}
    \usepackage{mathptmx}
    \DeclareMathAlphabet{\mathcal}{OMS}{cmsy}{m}{n}

\fi

\ifmoc
    \documentclass{mcom-l}

    \newtheorem{theorem}{Theorem}[section]
    \newtheorem{lemma}[theorem]{Lemma}
    \newtheorem{corollary}[theorem]{Corollary}

    \theoremstyle{definition}

    \newtheorem{algorithm}[theorem]{Algorithm}

    \theoremstyle{remark}
    \newtheorem{remark}[theorem]{Remark}
    \newtheorem{assumption}[theorem]{Assumption}
    \numberwithin{equation}{section}

\fi

\ifanm
    \documentclass[ ]{elsarticle}
    \journal{Applied Numerical Mathematics}

    \newtheorem{theorem}{Theorem}
    \newtheorem{lemma}[theorem]{Lemma}
    \newtheorem{corollary}[theorem]{Corollary}
    \newtheorem{remark}{Remark}
\fi

\ifjcam
    \documentclass[ ]{elsarticle}
    \journal{Journal of Computational and Applied Mathematics}

    \newtheorem{theorem}{Theorem}
    \newtheorem{lemma}[theorem]{Lemma}
    
    \newtheorem{remark}{Remark}
\fi

\ifjcm
    \documentclass{jcmlatex}

\fi

\ifsiam
    \usepackage[all]{xy}
    \usepackage{amssymb}
    \usepackage{braket,amsfonts,amsopn}
    \usepackage{makecell,multirow,diagbox}
    \usepackage{mathrsfs}
    \usepackage{graphicx,epstopdf}
    \usepackage{subfigure}
    \usepackage{stmaryrd}
    \numberwithin{theorem}{section}
    \numberwithin{equation}{section}
\else
    \usepackage{curves,amsmath,amssymb,fancybox,exscale,lineno,bm}
    \usepackage[all]{xy}
    \usepackage[pdftex]{color} 
    \usepackage[pdftex,colorlinks]{hyperref}
    \usepackage{esint}
    \usepackage{graphicx}
    \usepackage{stmaryrd}
    \numberwithin{equation}{section}
    \numberwithin{theorem}{section}
\fi

\definecolor{myblue}{rgb}{0.2,0.2,0.7}
\definecolor{mygreen}{rgb}{0,0.6,0}
\definecolor{mycyan}{rgb}{0,0.6,0.6}
\definecolor{myred}{rgb}{0.9,0.2,0.2}
\definecolor{mymagenta}{rgb}{0.9,0.2,0.9}

\definecolor{dark}{gray}{0.6}
\definecolor{light}{gray}{0.8}

\DeclareMathOperator{\tr}{tr}

\DeclareMathOperator{\osc}{osc}
\DeclareMathOperator{\gap}{gap}
\DeclareMathOperator{\tol}{tol}
\DeclareMathOperator{\supp}{supp}

\DeclareMathOperator{\up}{up}
\DeclareMathOperator{\low}{low}
\DeclareMathOperator{\loc}{loc}
\DeclareMathOperator{\opt}{opt}

\newcommand{\ab}[2]{\langle#1,#2\rangle}

\newcommand{\Eh}{\mathcal{E}_{h}}

\newcommand{\Rom}[1]{\uppercase\expandafter{\romannumeral#1}}
\newcommand{\Th}[0]{\mathcal{T}_{h}}
\newcommand{\lr}[1]{\llbracket#1\rrbracket}
\newcommand{\vertiii}[1]{{\left\vert\kern-0.25ex\left\vert\kern-0.25ex\left\vert #1 
    \right\vert\kern-0.25ex\right\vert\kern-0.25ex\right\vert}}
    
\def\rebAuthor{Randolph E. Bank}
\def\rebAuthor{Michael Holst}
\def\yulAuthor{Yuwen Li}

\ifsiam
    \def\rebShortAuthor{R. E. Bank}
    
    \def\yulShortAuthor{Y.~Li}
\else
    \def\rebShortAuthor{R. E. Bank}
    
    \def\yulShortAuthor{Y. Li}
\fi

\def\rebAddress{Department of Mathematics, University of California San Diego,
 La Jolla, California 92093-0112.}

\def\yulAddress{Department of Mathematics, University of California San Diego,
 La Jolla, California 92093-0112.}

\def\rebEmail{rbank@ucsd.edu}
\def\yulEmail{yul739@ucsd.edu}

\def\rebThanks{The work of this author was supported by the National
Science Foundation under  contract DMS-1318480.}

\def\yulThanks{The work of this author was supported by NSF.}

\title{Some convergence and optimality results of adaptive mixed methods in finite element exterior calculus}

\def\shortTitle{Convergence and optimality of AMFEM in FEEC}

\def\myKeywords{a posteriori error estimate, adaptive mixed finite element method, finite element exterior calculus, Hodge Laplacian, convergence,  optimality}

\def\myAMS{65N12, 65N15, 65N30, 65N50, 41A25}

\def\myAbstract{
In this paper, we present several new a posteriori error estimators and two adaptive mixed finite element methods \textsf{AMFEM1} and \textsf{AMFEM2} for the Hodge Laplacian problem in finite element exterior calculus. We prove that \textsf{AMFEM1} and \textsf{AMFEM2} are both convergent starting from any initial coarse mesh. In addition, we prove the quasi-optimality of \textsf{AMFEM2}. Comparing to existing literature, our results work on Lipschitz domains with nontrivial cohomology and provide the first norm convergence and quasi-optimality results for the Hodge Laplacian.
}

\begin{document}


\ifsiam
  \author{\yulAuthor%
         \thanks{\yulAddress {Email:}\, \yulEmail.}
}
  \maketitle

  \begin{abstract}\myAbstract\end{abstract}
  \pagestyle{myheadings}
  \thispagestyle{plain}
  \markboth{\yulShortAuthor}{\shortTitle}

\fi   


\ifnummat
   \author{
          \yulAuthor%
           }
  \institute{
              \yulShortAuthor : \yulAddress\, {Email:}\yulEmail}
  \date{Received: \today\  / Accepted: date}
  \maketitle
  \begin{abstract}\myAbstract\end{abstract}
  \begin{subclass}\myAMS \end{subclass}
  \markboth{
  \yulShortAuthor }{\shortTitle}
\fi


\ifcvs
   \author{\rebAuthor%
         \thanks{\rebShortAuthor : \rebThanks}
         \and
         \yulAuthor%
         \thanks{\yulShortAuthor : \yulThanks}
          }
  \institute{\rebShortAuthor : \rebAddress \, {Email:}\rebEmail \\ 
        \yulShortAuthor  : \yulAddress \, {Email:}\yulEmail }
  \date{Received: \today\  / Accepted: date}
  \maketitle
  \begin{abstract}\myAbstract\end{abstract}
  \begin{subclass}\myAMS \end{subclass}
\fi



\ifmoc
    \bibliographystyle{amsplain}


    \author[\yulShortAuthor]{\yulAuthor}
    \address{\yulAddress}
    \email{\yulEmail}

    \subjclass[2010]{Primary \myAMS}
    \date{\today}
    \begin{abstract}\myAbstract\end{abstract}
    \title[convergence and optimality of adaptive AMFEEC]{Some convergence and optimality results of adaptive mixed methods in finite element exterior calculus}
    \maketitle
\fi


\ifanm
  \begin{frontmatter}

  \author{\rebAuthor\fnref{fn1}}
  \address{\rebAddress\, {Email:}\rebEmail}
  \fntext[fn1]{\rebThanks} 

  \author{\yulAuthor\corref{cor1}\fnref{nf2}}
  \address{\yulAddress\, {Email:}\yulEmail} 
   \cortext[cor1]{Corresponding Author}
  \fntext[fn2]{\yulThanks} 

  \begin{abstract}\myAbstract\end{abstract}

  \end{frontmatter}

\fi


\ifjcam
  \begin{frontmatter}

  \author{\rebAuthor\fnref{fn1}}
  \address{\rebAddress\, {Email:}\rebEmail}
  \fntext[fn1]{\rebThanks} 

  \author{\yulAuthor\corref{cor1}\fnref{nf2}}
  \address{\yulAddress\, {Email:}\yulEmail} 
   \cortext[cor1]{Corresponding Author}
  \fntext[fn2]{\yulThanks} 

  \begin{abstract}\myAbstract\end{abstract}

  \end{frontmatter}

\fi


\ifjcm
 \markboth{\rebShortAuthor and \yulShortAuthor}
          {\shortTitle}

\author{\rebAuthor\footnote{\rebThanks}\thanks{\rebAddress \\ Email: \rebEmail}
\and
        \yulAuthor\footnote{\yulThanks}\thanks{\yulAddress \\ Email: \yulEmail}}

\maketitle

  \begin{abstract}\myAbstract\end{abstract}
\fi



\section{Introduction}

In this paper, we present some a posteriori error estimators, convergence and optimality results of adaptive finite element methods (FEM) for the Hodge Laplacian problem on Lipschitz domains in finite element exterior calculus (FEEC). Arnold, Falk, and Winther \cite{AFW2006} proposed the framework of FEEC for solving the Hodge Laplacian problem on the de Rham complex of differential forms, which is a wild generalization of the traditional mixed formulations for scalar and vector Laplacian problems in $\mathbb{R}^{2}$ and $\mathbb{R}^{3}$. In \cite{AFW2010}, they extended FEEC to closed Hilbert complexes. A priori error estimates were given under suitable regularity assumption on the exact solution and the domain. Such assumptions may fail at the presence of singularity, e.g., jump coefficients, nonsmooth sources, domains with nonsmooth boundary or reentrant corners. In these cases, standard FEMs on quasi-uniform meshes suffer from a slow rate of convergence. To overcome the singularity barrier, adaptive FEMs (AFEM) were proposed, see \cite{BS2001} for a thorough introduction. Convergence and optimality results of AFEMs for the primal formulation of scalar elliptic equations are extensive in the literature,   see \cite{Dorfler1996,MNS2000,BDD2004,Stevenson2007,CKNS2008,FFP2014} and references therein. Readers are also referred to  \cite{ZCSWX2012,CH2006,CHX2009,BM2008,HHX2011} for relevant results in AFEMs for Maxwell equations and adaptive mixed methods for scalar $2^{nd}$ and $4^{th}$ order elliptic equations. However, there are only a few papers on convergence and optimality of AFEMs in the FEEC framework, see, e.g., \cite{HLMS2019,CW2017} for the Hodge Laplace equation \eqref{compactHL} and \cite{Demlow2017} for computing harmonic forms. 

The first step of designing AFEMs is developing a posteriori error estimates. A posteriori error estimates in mixed finite element methods (MFEM) are technical and relying on delicate decomposition results, see, e.g., \cite{Alonso1996, CC1997, Schoberl2008, HX2011}. Until recently, Demlow and Hirani \cite{DH2014} constructed the first residual-type a posteriori error estimator controlling the coupled error $\|\sigma-\sigma_{h}\|_{H\Lambda^{k-1}}+\|u-u_{h}\|_{H\Lambda^{k}}+\|p-p_{h}\|$ of the mixed method \eqref{HL} for the Hodge Laplacian \eqref{compactHL}. {Using a correspondence between differential forms and scalar/vector fields in $\mathbb{R}^3$ (see page 14 in \cite{AFW2006}), the de Rham complex \eqref{dR} with $n=3$
can be identified with the well-known complex
\begin{equation}\label{dR3}
H^1(\Omega)\xrightarrow{\nabla} H(\text{curl},\Omega)\xrightarrow{\text{curl}}H(\text{div},\Omega)\xrightarrow{\text{div}}L^2(\Omega).
\end{equation}
In this case, the Hodge Laplacian \eqref{compactHL} in $\mathbb{R}^3$ reduces to the scalar Poisson's equation when $k=0, 3$ and vector Poisson's equation when $k=1,2$. The corresponding mixed method \eqref{HL} covers the classical Raviart--Thomas \cite{RT1977}, Brezzi--Douglas--Marini \cite{BDM1985}, and N\'ed\'elec edge element method \cite{Nedelec1980,Nedelec1986}, see page 60 in \cite{AFW2006}. In $\mathbb{R}^3,$ $\|\cdot\|_{H\Lambda^3}, \|\cdot\|$ are simply the $L^2$ norm and $\|\cdot\|_{H\Lambda^0}, \|\cdot\|_{H\Lambda^1}, \|\cdot\|_{H\Lambda^2}$ can be identified with the usual Sobolev norms $\|\cdot\|_{H^1}, \|\cdot\|_{H(\text{curl})}, \|\cdot\|_{H(\text{div})}$, respectively.} Important tools in \cite{DH2014} include a regular decomposition result and a commuting quasi-interpolation, see Theorem \ref{rdqi}. Relevant a posteriori error estimates and quasi-interpolations for the complex \eqref{dR3} have also been established by Sch\"oberl \cite{Schoberl2008}. In contrast to the estimator for the coupled error in \cite{DH2014}, we will derive several separate estimators for the decoupled errors $\|\sigma-\sigma_{h}\|_{H\Lambda^{k-1}}$, $\|p-p_{h}\|$, and $\|d(u-u_{h})\|$. All of these estimators are locally efficient while the error estimator in \cite{DH2014} has not been shown to be efficient because of the intractable term $\|P_{\mathfrak{H}}u_{h}\|$. Our separate estimator $\eta_{\sigma}(K)$ (see Theorem \ref{bdsigma}) is favorable when $\sigma$ is a practically relevant quantity. For example, the traditional MFEM for Poisson's equation is designed to obtain a better finite element solution $\sigma_{h}$ approximating $\sigma=\delta^nu=-\nabla u$. In addition, when using the pair $P_{r+1}\Lambda^{k-1}(\Th)\times P_{r}\Lambda^{k}(\Th)$ to discretize $H\Lambda^{k-1}(\Omega)\times H\Lambda^{k}(\Omega)$ (see \cite{AFW2006,AFW2010}), we have the a priori estimate 
$$\|\sigma-\sigma_{h}\|_{H\Lambda^{k-1}}+\|u-u_{h}\|_{H\Lambda^{k}}+\|p-p_{h}\|=O(h^{r}),$$
as well as the improved a priori estimates $$\|\sigma-\sigma_{h}\|_{H\Lambda^{k-1}}=O(h^{r+1}),\quad\|\sigma-\sigma_{h}\|=O(h^{r+2}),\quad\|d(u-u_{h})\|=O(h^{r}).$$
In this case, $\|d(u-u_{h})\|$ dominates the coupled error, which makes a coupled error estimator inefficient for the purpose of estimating $\sigma-\sigma_{h}$.

The authors in \cite{HLMS2019} developed a quasi-optimal adaptive MFEM (AMFEM) for the Hodge Laplacian with index $k=n$, which is in fact the traditional MFEM for Poisson's equation on domains with general topology in $\mathbb{R}^n$. \cite{CW2017} seems to be the only convergence result of AMFEM for the Hodge Laplacian with general index $1\leq k\leq n-1$. In particular, Chen and Wu \cite{CW2017} designed a convergent AMFEM on domains with vanishing $k$th-cohomology ($\mathfrak{H}^k=\{0\}$), e.g., a contractible domain. Their algorithm can reduce $\|d(\sigma-\sigma_{h})\|^2+\|d(u-u_{h})\|^2$ below any given error tolerance in finite steps. By assuming fineness of the initial mesh, they also claimed quasi-optimality of their algorithm. Since $d$ has a large kernel, \cite{CW2017} in fact provides semi-norm convergence, i.e.,  $\|d(\sigma-\sigma_{h})\|^2+\|d(u-u_{h})\|^2\rightarrow0$ does not imply $\sigma_h\rightarrow\sigma$ nor $u_h\rightarrow u$ in the $L^{2}\Lambda$, $H\Lambda$, or any other Sobolev norm. In view of literature on traditional mixed methods and MFEMs in FEEC, researchers are more concerned with errors measured in the $L^{2}$-norm $\|\cdot\|$ or the $V$-norm $\|\cdot\|_{H\Lambda}$. For example, the authors in \cite{AFW2010,AL2017} made a great effort to obtain improved decoupled a priori error estimates for $\|\sigma-\sigma_{h}\|$ and $\|u-u_{h}\|$. 

The contributions of this paper are listed as follows. All of our results are in full generality, i.e., they work for the Hodge Laplace equation with arbitrary index $1\leq k\leq n$, dimension $n\geq2$, and domain $\Omega$ with general topology $(\mathfrak{H}^{k}\neq\{0\})$.
\begin{enumerate}
\item We derive reliable and efficient a posteriori error estimators $\eta_{\sigma}, \eta_{p}$, and $\eta_{du}$ for controlling $\|\sigma-\sigma_{h}\|_{H\Lambda^{k-1}}$, $\|p-p_{h}\|$, and  $\|d(u-u_{h})\|$, respectively. Although $\eta_{du}$ appears in \cite{CW2017} in the case $\mathfrak{H}^{k}=\{0\}$, the estimators $\eta_{\sigma}$ and $\eta_{p}$ are new and building blocks of \textsf{AMFEM1} and \textsf{AMFEM2}.

\item We develop a convergent adaptive algorithm \textsf{AMFEM1} for controlling $\|\sigma-\sigma_{h}\|^2_{H\Lambda^{k-1}}$+$\|p-p_{h}\|^2$+$\|d(u-u_{h})\|^2$. Prior to this paper, we are not aware of any norm convergence result (w.r.t.~$\|\sigma-\sigma_{h}\|_{H\Lambda^{k-1}}$, $\|\sigma-\sigma_{h}\|$, $\|u-u_{h}\|_{H\Lambda^{k}}$,  $\|u-u_{h}\|$ or any combination of them) of AMFEM for the Hodge Laplacian with index $1\leq k\leq n-1$. In addition, \textsf{AMFEM1} with $k=n$ gives new result on AMFEM for Poisson's equation, see Remark \ref{kn2}. 

\item By dropping several marking steps in \textsf{AMFEM1}, we obtain a quasi-optimal adaptive method \textsf{AMFEM2} controlling $\|\sigma-\sigma_{h}\|_{H\Lambda^{k-1}}$. The new ingredient of the optimality proof is a localized upper bound. As far as we know, there is no result on quasi-optimality of AMFEMs for the Hodge Laplacian with index $1\leq k\leq n-1$ in the literature. 
\end{enumerate}

Our convergence analysis of \textsf{AMFEM1} basically follows the framework of \cite{CKNS2008}. The ingredients include global reliablity of $\eta_{\sigma}$, $\eta_{p}$, $\eta_{du}$, the error reduction in Lemma \ref{lemma1}, and the estimator reduction in Lemma \ref{lemma2}. Proofs of a posteriori upper bounds rely on the Hodge decomposition, the regular decomposition, and a gap estimate in Lemma \ref{gaph}. The role of Lemma \ref{lemma1} is similar to the orthogonality or quasi-orthogonality in convergence analysis of AFEMs, see, e.g., \cite{MN2005,CHX2009,BM2008,FFP2014}. The estimator reduction in Lemma \ref{lemma2} results from dependence of our estimators on $\sigma_{h}, p_{h}$ and $du_{h}$. In contrast, the coupled error estimator in \cite{DH2014} depends on the triple $(\sigma_{h},u_{h},p_{h})$, in particular on $\mu\|u_{h}\|$ (see \eqref{errorDH}), which seems an obstacle against proving estimator reduction. The contraction on $\|\sigma-\sigma\|_{H\Lambda^{k-1}}$ in \textsf{AMFEM1} is realized by discovering the suitably weighted quasi-error $\|\sigma-\sigma_h\|^2+\zeta\|d(\sigma-\sigma_h)\|^2+\rho_\sigma\eta_{\sigma}^{2}(\Th)$, see the proof of Theorem \ref{convergenceAMFEM1}. Convergence of $\|p-p_h\|$ and $\|d(u-u_h)\|$ is guaranteed by separate D\"orfler markings \eqref{mark2} and \eqref{mark3}. 

Separate markings may cause problems when proving optimality, see Section 6 in \cite{CKNS2008} for details. Hence we consider the standard adaptive method \textsf{AMFEM2} by dropping markings \eqref{mark2} and \eqref{mark3}. \textsf{AMFEM2} seems to be the first quasi-optimal AMFEM controlling $\|\sigma-\sigma_{h}\|_{H\Lambda^{k-1}}$ for the Hodge Laplace equation or even Poisson's equation when $k=n$. In the contraction proof of \textsf{AMFEM2}, there is no smallness assumption on the initial mesh size, in contrast to some AFEMs for similar problems posed on the de Rham complex, e.g., Maxwell equations in $\mathbb{R}^3$ \cite{ZCSWX2012}. In fact, contraction of \textsf{AMFEM2} is guaranteed by utilizing the aforementioned quasi-error. See also Remark \ref{kn2} for a comparison between \textsf{AMFEM2} and existing quasi-optimal AMFEMs for Poisson's equation. 

The new ingredient in the optimality proof is the localized upper bound in Theorem \ref{disupper}, which is a discrete and local version of the global reliability in Theorem \ref{bdsigma}. The proof is similar to Theorem \ref{bdsigma} in the sense that the discrete Hodge decomposition instead of the Hodge decomposition is applied. To localize the upper bound, we use the local bounded cochain projection developed by Falk and Winther \cite{FW2014} and the classical Scott--Zhang interpolation \cite{SZ1990}. Demlow (see Lemma 3 in \cite{Demlow2017}) originally used this technique to prove a localized upper bound in the AFEM for computing harmonic forms.

The rest of this paper is organized as follows. In Section \ref{pre}, we introduce the background of FEEC. In Section \ref{aposteriori}, we derive several a posteriori error estimates. Section \ref{convergence} is devoted to  convergence analysis of  \textsf{AMFEM1} and \textsf{AMFEM2}. In Section \ref{optimality}, we construct a localized upper bound and prove the quasi-optimality of \textsf{AMFEM2}. 
  
\section{Preliminaries}\label{pre}
In this section, we follow the convention of \cite{AFW2010,AFW2006} to introduce necessary background of FEEC in our analysis.
\subsection{Closed Hilbert complex}
Consider the closed Hilbert complex $(W,d):$ 
\begin{equation*}
\cdots\rightarrow W^{k-1}\xrightarrow{d^{k-1}}W^{k}\xrightarrow{d^{k}}W^{k+1}\xrightarrow{d^{k+1}}\cdots,
\end{equation*}
i.e., for each index $k$, $W^{k}$ is a Hilbert space equipped with the inner product $\ab{\cdot}{\cdot}=\ab{\cdot}{\cdot}_{W^k}$, $d^{k}: W^{k}\rightarrow W^{k+1}$ is a densely defined unbounded and closed operator, the range of $d^{k}$ is closed in $W^{k+1}$ and contained in the domain of $d^{k+1}$ and $d^{k+1}\circ d^{k}=0$. Let $\mathfrak{Z}^{k}=N(d^{k})$ denote the null space of $d^{k}$, $\mathfrak{B}^{k}=R(d^{k-1})$ the range of $d^{k-1}$, and $\mathfrak{H}^{k}=\mathfrak{Z}^{k}\cap\mathfrak{B}^{k\perp}$ the space of harmonic forms, where $\perp$ is the notation of orthogonal complement w.r.t.~$\ab{\cdot}{\cdot}$ in $W^{k}$. $(W,d)$ admits the Hodge decomposition
\begin{equation}\label{Hodge1}
W^{k}=\mathfrak{B}^{k}\oplus\mathfrak{H}^{k}\oplus\mathfrak{Z}^{k\perp},
\end{equation} 
where $\oplus$ denotes the direct sum w.r.t.~$\ab{\cdot}{\cdot}$. Correspondingly, let $P_{\mathfrak{B}}$, $ P_{\mathfrak{H}}$, and $P_{\mathfrak{Z}^{\perp}}$ be projections onto $\mathfrak{B}^{k}$,  $\mathfrak{H}^{k}$, and $\mathfrak{Z}^{k\perp}$ w.r.t.~$\ab{\cdot}{\cdot}$, respectively.
The complex $(W,d)$ is related to the dual complex $(W,d^{*})$: 
\begin{equation*}
\cdots\rightarrow W^{k+1}\xrightarrow{d_{k+1}^{*}}W^{k}\xrightarrow{d^{*}_{k}}W^{k-1}\xrightarrow{d^{*}_{k-1}}\cdots
\end{equation*}
where $d_{k}^{*}$ is the adjoint of $d^{k-1}$. $(W,d^{*})$ is also a closed Hilbert complex. Let $\mathfrak{Z}^{*}_{k}=N(d_{k}^{*})$ and $\mathfrak{B}_{k}^{*}=R(d_{k+1}^{*})$. By the closed range theorem, $\mathfrak{Z}^{k\perp}=\mathfrak{B}_{k}^{*}$ and the Hodge decomposition \eqref{Hodge1} can be written as 
\begin{equation}\label{Hodge2}
W^{k}=\mathfrak{B}^{k}\oplus\mathfrak{H}^{k}\oplus\mathfrak{B}_{k}^{*}.
\end{equation}
Associated with any Hilbert complex $(W,d)$ is the domain complex $(V,d)$:
\begin{equation*}
\cdots\rightarrow V^{k-1}\xrightarrow{d^{k-1}}V^{k}\xrightarrow{d^{k}}V^{k+1}\xrightarrow{d^{k+1}}\cdots,
\end{equation*}
where $V^{k}$ is the domain of $d^{k}$ in $W^{k}$, equipped with the inner product $\ab{u}{v}_{V}=\ab{u}{v}_{V^{k}}:=\ab{u}{v}_{W^{k}}+\ab{d^{k}u}{d^{k}v}_{W^{k+1}}.$ The domain complex of $(W,d^{*})$ is $(V^{*},d^{*})$, where $V^{*}_{k}$ is the domain of $d_{k}^{*}$. We use $\|\cdot\|$ to denote the norm of $W$ and $\|\cdot\|_{V}$ the norm of $V$. By inverse mapping theorem, the  Poincar\'e inequality $\|v\|\leq c_P\|d^{k}v\|$ holds, provided $v\in\mathfrak{Z}^{k\perp V}:=\mathfrak{Z}^{k\perp}\cap V^{k}$. For the dual complex $(V^{*},d^{*})$, a Poincar\'e inequality still holds with the same constant, i.e., $\|v\|\leq c_P\|d_{k+1}^{*}v\|$ provided $v\in\mathfrak{Z}_{k+1}^{*\perp V^{*}}:=\mathfrak{Z}_{k+1}^{*\perp}\cap V_{k+1}^{*}$. In fact, let $v=d^{k}w\in\mathfrak{Z}_{k+1}^{*\perp}=\mathfrak{B}^{k+1}$ with $w\in\mathfrak{Z}^{k\perp V}$, 
$$\|v\|^{2}=\ab{d_{k+1}^{*}v}{w}\leq\|d_{k+1}^{*}v\|\|w\|\leq c_P\|d_{k+1}^{*}v\| \|d^{k}w\|=c_P\|d_{k+1}^{*}v\|\|v\|.$$

Given $v\in W^{k}$, by the Hodge decomposition \eqref{Hodge2}, there exist $v_{1}\in V^{k-1},\ v_{2}\in V_{k+1}^{*},$ and $q\in\mathfrak{H}^{k}$, such that $v=d^{k-1}v_{1}+d_{k+1}^{*}v_{2}+q$. Clearly, we can assume $v_{1}\in(\mathfrak{Z}^{k-1})^{\perp V},\ v_{2}\in\mathfrak{Z}_{k+1}^{*\perp V^{*}}$, and $\|v_{1}\|\leq c_P\|d^{k-1}v_{1}\|$, $\|v_{2}\|\leq c_P\|d_{k+1}^{*}v_{2}\|$. This decomposition will be applied in proofs of many theorems in this paper.

Given a Hilbert complex $(W,d)$, the unbounded operator $L=dd^{*}+d^{*}d: W^{k}\rightarrow W^{k}$ is called the abstract Hodge Laplacian. The abstract Hodge Laplacian problem is to solve $Lu=f$ mod $\mathfrak{H}^{k}$. The variational formulation is to find $(u,p)\in(V^{k}\cap V_{k}^{*})\times\mathfrak{H}^{k}$, such that
\begin{equation}\label{primalHL}
\begin{aligned}
\ab{du}{dv}+\ab{d^{*}u}{d^{*}v}+\ab{v}{p}&=\ab{f}{v},\quad &&v\in V^{k}\cap V_{k}^{*},\\
\ab{u}{q}&=0,  &&q\in\mathfrak{H}^{k}.
\end{aligned}
\end{equation}
However, it is generally difficult to construct a finite element subspace of $V^{k}\cap V_{k}^{*}$ in practice. Alternatively, Arnold, Falk, and Winther \cite{AFW2010} considered the equivalent and well-posed mixed formulation of the abstract Hodge Laplacian: find $(\sigma,u,p)\in V^{k-1}\times V^{k}\times\mathfrak{H}^{k}$, such that
\begin{equation}\label{HL}
\begin{aligned}
\ab{\sigma}{\tau}-\ab{d\tau}{u}&=0,\quad&&\tau\in V^{k-1},\\
\ab{d\sigma}{v}+\ab{du}{dv}+\ab{v}{p}&=\ab{f}{v},&& v\in V^{k},\\
\ab{u}{q}&=0,&&q\in\mathfrak{H}^{k}.
\end{aligned}
\end{equation}

\subsection{Approximation of Hilbert complex}
For each index $k$, choose a finite dimensional subspace $V_{h}^{k}$ of $V^{k}$. We assume that $dV_{h}^{k}\subset V_{h}^{k+1}$ so that $(V_{h},d)$ is a subcomplex of $(V,d)$. We take $W_{h}^{k}$ to be the same space $V_{h}^{k}$ but equipped with the inner product $\ab{\cdot}{\cdot}$. For the restricted exterior derivative $d: V_{h}^{k}\rightarrow V_{h}^{k+1}$, the adjoint $d_{h}^{*}: V_{h}^{k+1}\rightarrow V_{h}^{k}$ is defined as 
$$\ab{d_{h}^{*}u}{v}=\ab{u}{dv},\quad u\in V_{h}^{k+1},\ v\in V_{h}^{k}.$$  Following \cite{AFW2010}, we use $\mathfrak{Z}_{h},\ \mathfrak{B}_{h},\ \mathfrak{H}_{h},\ \mathfrak{Z}_{h}^{*},\ \mathfrak{B}_{h}^{*}=\mathfrak{Z}_{h}^{\perp}$ etc.~for obvious meanings (cf.\cite{AFW2006} for details). Note that $\mathfrak{B}_{h}^{k}\subset\mathfrak{B}^{k}$, $\mathfrak{Z}_{h}^{k}\subset\mathfrak{Z}^{k}$, but in general $\mathfrak{H}_{h}^{k}\not\subseteq\mathfrak{H}^{k}$, $\mathfrak{Z}_{h}^{k\perp}\not\subseteq\mathfrak{Z}^{k\perp}$. The discrete Hodge decomposition holds:
\begin{equation}\label{disHodge}
V_{h}^{k}=\mathfrak{B}_{h}^{k}\oplus\mathfrak{H}^{k}_{h}\oplus\mathfrak{Z}_{h}^{k\perp}=\mathfrak{B}_{h}^{k}\oplus\mathfrak{H}^{k}_{h}\oplus\mathfrak{B}_{h}^{*}.
\end{equation}
Recall that $d_{h}^{*}$, $\oplus$, and $\perp$ are all based on the $W$-inner product $\ab{\cdot}{\cdot}$. In order to understand how well the approximation of $V$ by $V_{h}^{k}$ is, it is necessary to construct a bounded cochain projection $\pi_{h}$ from $(V,d)$ to $(V_{h},d)$. To be precise, for each index $k$, $\pi_{h}^{k}$ maps $V^{k}$ onto $V_{h}^{k}$, $\pi_{h}^{k}|_{V_{h}^{k}}=$id, $d^{k}\pi_{h}^{k}=\pi_{h}^{k+1}d^{k}$, and $\|\pi_{h}^{k}\|_{V}=\|\pi_{h}^{k}\|_{V^{k}\rightarrow V_{h}^{k}}<\infty$ uniformly with respect to the discretization parameter $h$. Theorem 3.6 in \cite{AFW2010} gives the discrete Poincar\'e inequality $\|v\|\leq c_P\|\pi_{h}^{k}\|_{V}\|dv\|$ provided $v\in\mathfrak{Z}_{h}^{\perp}$. Similar to the continuous case, another discrete Poincar\'e inequality $\|v\|\leq c_P\|\pi_{h}^{k}\|_{V}\|d_{h}^{*}v\|$ holds provided $v\in\mathfrak{Z}_{h}^{*\perp}=\mathfrak{B}_{h}$.

Using the discrete complex $(V_{h},d)$, the discrete mixed formulation for  \eqref{HL} is to find $(\sigma_{h},u_{h},p_{h})\in V_{h}^{k-1}\times V_{h}^{k}\times\mathfrak{H}_{h}^{k}$, such that
\begin{equation}\label{DHL}
\begin{aligned}
\ab{\sigma_{h}}{\tau}-\ab{d\tau}{u_{h}}&=0,\quad&&\tau\in V_{h}^{k-1},\\
\ab{d\sigma_{h}}{v}+\ab{du_{h}}{dv}+\ab{v}{p_{h}}&=\ab{f}{v},&&v\in V_{h}^{k},\\
\ab{u_{h}}{q}&=0,&&q\in\mathfrak{H}_{h}^{k}.
\end{aligned}
\end{equation}
 Since $\mathfrak{H}_{h}^{k}$ is not contained in $\mathfrak{H}^{k}$, \eqref{DHL}
is not a standard Galerkin method. The authors in \cite{AFW2010} proved the discrete inf-sup condition and a priori error estimates of \eqref{DHL}. Combining \eqref{HL} and \eqref{DHL}, we obtain the error equation
\begin{subequations}\label{error}
\begin{align}
\ab{\sigma-\sigma_{h}}{\tau}-\ab{d\tau}{u-u_{h}}&=0,\label{error1}\\
\ab{d(\sigma-\sigma_{h})}{v}+\ab{d(u-u_{h})}{dv}+\ab{v}{p-p_{h}}&=0,\label{error2}
\end{align}
\end{subequations}
for all $(\tau,v)\in V_{h}^{k-1}\times V_{h}^{k}$.

\subsection{De Rham complex and approximation}
The de Rham complex is a canonical example of the closed Hilbert complex $(W,d)$. Let $\Omega\subset\mathbb{R}^{n}$ be a bounded Lipschitz domain. Given $0\leq k\leq n$, let $\Lambda^{k}(\Omega)$ denote the space of smooth $k$-forms $\omega$ which can be written as 
$$\omega=\sum_{1\leq\alpha_{1}<\cdots<\alpha_{k}\leq n}\omega_{\alpha}dx^{\alpha_{1}}\wedge\cdots\wedge dx^{\alpha_{k}},$$
where the coefficient $\omega_{\alpha}\in C^{\infty}(\Omega)$ and $\wedge$ is the wedge product. The space $\Lambda^{k}(\Omega)$ is naturally endowed with the $L^{2}$-inner product $\ab{\cdot}{\cdot}$ and norm $\|\cdot\|$.
Let $d^k: \Lambda^{k}(\Omega)\rightarrow\Lambda^{k+1}(\Omega)$ denote the exterior derivative for differential forms. The Sobolev version of $\Lambda^{k}(\Omega)$ is $L^{2}\Lambda^{k}(\Omega)$, the completion of $\Lambda^{k}(\Omega)$ under $\|\cdot\|$, which is simply the space of $k$-forms with $L^2$ coefficients. Corresponding to $(W,d)$ is the $L^{2}$ complex 
\begin{equation*}
L^{2}\Lambda^{0}(\Omega)\xrightarrow{d^0} L^{2}\Lambda^{1}(\Omega)\xrightarrow{d^1}\cdots\xrightarrow{d^{n-2}} L^{2}\Lambda^{n-1}(\Omega)\xrightarrow{d^{n-1}} L^{2}\Lambda^{n}(\Omega),
\end{equation*}
where $d^k$ is the weak exterior derivative viewed as a densely defined unbounded operator with domain $H\Lambda^{k}(\Omega):=\{\omega\in L^{2}\Lambda^{k}(\Omega): d\omega\in L^{2}\Lambda^{k+1}(\Omega)\}.$ The $L^2$ de Rham complex [corresponds to ($V,d$)] is 
\begin{equation}\label{dR}
H\Lambda^0(\Omega)\xrightarrow{d^0} H\Lambda^1(\Omega)\xrightarrow{d^1}\cdots\xrightarrow{d^{n-2}} H\Lambda^{n-1}(\Omega)\xrightarrow{d^{n-1}} H\Lambda^{n}(\Omega).
\end{equation}
To define the adjoint of $d$, we need the Hodge star operator $\star: \Lambda^{k}(\Omega)\rightarrow\Lambda^{n-k}(\Omega)$ determined by 
$$\int_{\Omega}\omega\wedge\mu=\ab{\star\omega}{\mu},\quad\mu\in\Lambda^{n-k}(\Omega).$$ The coderivative $\delta^k: \Lambda^{k}(\Omega)\rightarrow\Lambda^{k-1}(\Omega)$ is then defined as $\star\delta^k\omega=(-1)^{k}d^{n-k}\star\omega$. We may drop the index of $\delta$ and $d$ provided no confusion arises. $d$ and $\delta$ are related by the integrating by parts formula
\begin{equation}\label{IP}
\ab{d\omega}{\mu}=\ab{\omega}{\delta\mu}+\int_{\partial\Omega}\tr\omega\wedge\tr\star\mu,\quad\omega\in\Lambda^{k}(\Omega),\quad\mu\in\Lambda^{k+1}(\Omega),
\end{equation}
where the trace operator $\tr=i^{*}$ on $\partial\Omega$ is the pullback induced by the inclusion $i:\partial\Omega\rightarrow\overline{\Omega}$.
\eqref{IP} holds on any Lipschitz subdomain $\Omega_{0}\subset\Omega$ and in this case $\tr$ denotes the trace on $\partial\Omega_{0}$ by abuse of notation. The adjoint of $d^{k-1}$ is the unbounded operator $\delta^{k}: L^2\Lambda^{k}(\Omega)\rightarrow L^2\Lambda^{k-1}(\Omega)$ with domain 
$\mathring{H}^{*}\Lambda^{k}(\Omega):=\{\omega\in L^{2}\Lambda^{k}(\Omega): \delta\omega\in L^{2}\Lambda^{k-1}(\Omega),\ \tr\star
\omega=0\text{ on }\partial\Omega\}$. Then
\begin{equation*}
L^2\Lambda^{n}(\Omega)\xrightarrow{\delta^n} L^2\Lambda^{n-1}(\Omega)\xrightarrow{\delta^{n-1}}\cdots\xrightarrow{\delta^1} L^2\Lambda^{1}(\Omega)\xrightarrow{\delta^0}L^2\Lambda^{0}(\Omega)
\end{equation*}
is an example of $(W,d^*)$ and
\begin{equation*}
\mathring{H}^{*}\Lambda^{n}(\Omega)\xrightarrow{\delta^n} \mathring{H}^{*}\Lambda^{n-1}(\Omega)\xrightarrow{\delta^{n-1}}\cdots\xrightarrow{\delta^1} \mathring{H}^{*}\Lambda^{1}(\Omega)\xrightarrow{\delta^0}\mathring{H}^{*}\Lambda^{0}(\Omega)
\end{equation*}
is the domain complex [corresponds to $(V^{*},d^{*})$]. 
The Hodge Laplacian problem $(d\delta+\delta d)u=f$ mod $\mathfrak{H}^{k}$ is a concrete example of the abstract Hodge Laplacian. The corresponding mixed formulation is \eqref{HL} with $\delta$ replacing the abstract adjoint $d^{*}$ and $H\Lambda^{k-1}(\Omega)\times H\Lambda^k(\Omega)$ replacing $V^{k-1}\times V^k$. A more compact form is 
\begin{equation}\label{compactHL}
\begin{aligned}
\sigma=\delta u,\quad d\sigma+\delta d u+p=f,\quad u\perp\mathfrak{H}^{k}.
\end{aligned}
\end{equation}
Since $u\in\text{Dom}(\delta)=\mathring{H}^*\Lambda^k(\Omega)$ and $du\in \text{Dom}(\delta)=\mathring{H}^*\Lambda^{k+1}(\Omega)$, the boundary conditions $\tr\star u=0$, $\tr\star du=0$ on $\partial\Omega$ are implicitly posed. 

Let $\Th$ be a simplicial triangulation of $\Omega$. Let $h_K:=|K|^{\frac{1}{n}}$ denote the size of a simplex $K\in\Th,$
where $|K|$ is the volume of $K$. We still use $V_h^k$ to denote a suitable finite element space of $k$-forms with piecewise polynomial coefficients on $\Th$. Arnold, Falk, and Winther \cite{AFW2006} gave four possible choices of $V_{h}^{k-1}\times V_{h}^{k}$ for discretizing the pair $H\Lambda^{k-1}(\Omega)\times H\Lambda^{k}(\Omega)$. For example, $V_h^{k-1}\times V_h^k=\mathcal{P}_{r+1}\Lambda^{k-1}(\Th)\times\mathcal{P}_r\Lambda^k(\Th)$ with $r\geq0$, where $\mathcal{P}_i\Lambda^j(\Th)$ is the space of piecewise polynomial $j$-forms of degree $\leq i.$ We assume that $\Th$ is shape regular, i.e., $\max_{K\in\Th}\frac{r_{K}}{\rho_{K}}\leq C_{\Th}<\infty$, where $r_{K}$ and $\rho_{K}$ are radii of circumscribed and inscribed spheres of the simplex $K$, respectively. The constant $C_{\Th}$ quantifies the shape regularity of $\Th$ and should be uniformly bounded w.r.t.~the parameter $h$. For the construction of uniformly bounded cochain projection under shape regular but non quasi-uniform meshes, readers are referred to \cite{CW2008}.

Let $H^{s}\Lambda^{k}(\Omega)$ be the space of $k$-forms whose coefficients are in $H^{s}(\Omega)$ and be endowed with the $H^{s}$ Sobolev inner product and norm. Formula \eqref{IP} still holds for $\omega\in H^{1}\Lambda^{k}$ and $\mu\in H^{1}\Lambda^{k+1}.$ Consider $H^{1}\Lambda^{k}(\Th):=\{\omega\in L^{2}\Lambda^{k}(\Omega): \omega|_{K}\in H^{1}\Lambda^{k}(K)\ \text{for\ all\ } K\in\Th\}$, the space of piecewise $H^{1}$ $k$-forms. Let $\Eh$ be the set of $(n-1)$-faces in $\Th$. For each interior face $e\in\Eh$ and $\omega\in H^{1}\Lambda^k(\Th)$, let $\lr{\tr\omega}:=\tr(\omega|_{K_{1}})-\tr(\omega|_{K_{2}})$ denote the trace jump of $\omega$ on $e$, where the two adjacent simplices $K_{1}$ and $K_{2}$ share $e$ as an $(n-1)$-face. For each boundary face $e$, $\lr{\tr\omega}:=\tr(\omega|_{K})$ on $e$, where $K$ is the simplex having $e$ as a face. We use the notation $A\lesssim B$ provided $A\leq C\cdot B$ and $C$ is a generic constant depending only on $\Omega$ and shape regularity of the underlying mesh. We use $\|\cdot\|_{H^1}$ and $\|\cdot\|_{H\Lambda^k}$ to denote the $H^1\Lambda^l(\Omega)$ norm with some $l$ and $H\Lambda^k(\Omega)$ norm, respectively. $\ab{\cdot}{\cdot}_K$ denotes the $L^2$ inner product restricted to $K$. $\|\cdot\|_{K}$ and $\|\cdot\|_{\partial K}$ denote the $L^{2}$ norm restricted to $K$ and $\partial K$, respectively. 

To derive our error estimators and prove global reliability as well as localized upper bound, we need the regular decomposition and commuting quasi-interpolation developed by Demlow and Hirani, see Lemma 5 and Lemma 6 in \cite{DH2014}. 
\begin{theorem}[regular decomposition and quasi-interpolation]\label{rdqi}
Assume that $\Omega$ is a bounded Lipschitz domain in $\mathbb{R}^{n}$. Given $v\in H\Lambda^{k}(\Omega)$ with $1\leq k\leq n-1$, there exist $\varphi\in H^{1}\Lambda^{k-1}(\Omega)$ and $z\in H^{1}\Lambda^{k}(\Omega)$, such that $v=d\varphi+z$ and
$$\|\varphi\|_{H^{1}}+\|z\|_{H^{1}}\lesssim\|v\|_{H\Lambda^{k}}.$$
{When $k=0,$ $H\Lambda^0(\Omega)$ is identified with $H^1(\Omega)$ and the above decomposition trivially holds with $\varphi=0, z=v.$}
Let $\Pi_{h}^{k}: L^{2}\Lambda^{k}(\Omega)\rightarrow V_{h}^{k}$ be the commuting quasi-interpolation in \cite{DH2014}, then $\Pi_{h}^{k+1}\circ d^{k}=d^{k}\circ\Pi_{h}^{k}$ and the following approximation property holds. If $k=0$, we have 
\begin{equation*}
\sum_{K\in\Th}\big(h_{K}^{-2}\|v-\Pi_{h}^{k}v\|_{K}^{2}+h_{K}^{-1}\|v-\Pi_{h}^{k}v\|_{\partial K}^{2}+|v-\Pi_{h}^{k}v|_{H^{1}(K)}^{2}\big)\lesssim\|v\|^{2}_{H\Lambda^{k}}.
\end{equation*}
    If $1\leq k\leq n-1$, we have
    \begin{equation*}
       \begin{aligned} \sum_{K\in\Th}&\big(h_{K}^{-2}\|\varphi-\Pi^{k}_{h}\varphi\|_{K}^{2}+h_{K}^{-2}\|z-\Pi^{k}_{h}z\|_{K}^{2}\\
       &+h_{K}^{-1}\|\tr(\varphi-\Pi^{k}_{h}\varphi)\|_{\partial K}^{2}+h_{K}^{-1}\|\tr(z-\Pi^{k}_{h}z)\|_{\partial K}^{2}\big)\lesssim \|v\|^{2}_{H\Lambda^{k}}.
       \end{aligned}
    \end{equation*}
\end{theorem}
Similar decomposition holds for the space $\mathring{H}\Lambda^{k}(\Omega)=\{v\in H\Lambda^k(\Omega): \tr v=0\text{ on }\partial\Omega\}$. 

The discrete Hodge Laplacian \eqref{DHL} with $k=0$ reduces to the standard FEM for solving Poisson's equation under the homogeneous Neumann boundary condition. Hence we focus on the index $1\leq k\leq n$. To avoid cumbersome notations, we may drop the index $k$ appearing in $ \mathfrak{Z}^{k}, \mathfrak{B}^{k}, \mathfrak{H}^{k}, \Pi_h^k$ etc.~provided no confusion arises. In the end of this section, recall the four Poincar\'e inequalities 
\begin{equation*}
\begin{aligned} 
&\|v_{1}\|\leq c_P\|dv_{1}\|,\quad v_{1}\in\mathfrak{Z}^{\perp V},\quad \|v_{2}\|\leq c_P\|\delta v_{2}\|,\quad v_{2}\in\mathfrak{Z}^{*\perp V^{*}},\\
&\|v_{h}^{1}\|\leq c_P\|\pi_{h}\|_{V}\|dv_{h}^{1}\|,\quad v_{h}^{1}\in\mathfrak{Z}_{h}^{\perp},\quad\|v_{h}^{2}\|\leq c_P\|\pi_{h}\|_{V}\|\delta_{h} v_{h}^{2}\|,\quad v_{h}^{2}\in\mathfrak{Z}_{h}^{*\perp},
\end{aligned}
\end{equation*}
which are universal in the rest of this paper.

\section{A posteriori error estimates}\label{aposteriori}
In this section, we derive a posteriori estimates for $\|\sigma-\sigma_{h}\|_{H\Lambda^{k-1}},\  \|p-p_{h}\|$, and $ \|d(u-u_{h})\|$ in order and separately. An error estimator for $\|d(\sigma-\sigma_{h})\|$ was developed in \cite{CW2017}.
\begin{theorem}[Chen and Wu's a posteriori estimate for $\|d(\sigma-\sigma_{h})\|$]\label{bddsigma}
For $1\leq k\leq n-1$ and $f\in H^{1}\Lambda^{k}(\Th)$, it holds that
\begin{equation*}
\|d(\sigma-\sigma_{h})\|^2\lesssim\eta^2_{d\sigma}(\Th)=\sum_{K\in\Th}\eta^{2}_{d\sigma}(K),
\end{equation*}
where
\begin{equation*}
\eta^{2}_{d\sigma}(K)=h_{K}^{2}\|\delta(f-d\sigma_{h})\|_{K}^{2}+h_{K}\|\lr{\tr\star(f-d\sigma_{h})}\|_{\partial K}^{2}.
\end{equation*}
\end{theorem}
Theorem \ref{bddsigma} holds on any Lipschitz domain $\Omega$ with nontrivial cohomology although \cite{CW2017} is concerned with domains without harmonic forms.

To derive upper bounds for other errors in the mixed method \eqref{DHL}, the Hodge decomposition \eqref{Hodge2} is useful. In order to control the harmonic component in the decomposition, we need to estimate the gap between $\mathfrak{H}$ and $\mathfrak{H}_{h}$:
\begin{equation*}
\begin{aligned}
\delta(\mathfrak{H},\mathfrak{H}_{h})&:=\sup_{\|q\|=1,q\in\mathfrak{H}}\|q-P_{\mathfrak{H}_{h}}q\|,\\
\delta(\mathfrak{H}_{h},\mathfrak{H})&:=\sup_{\|q\|=1,q\in\mathfrak{H}_{h}}\|q-P_{\mathfrak{H}}q\|,\\
\gap(\mathfrak{H},\mathfrak{H}_{h})&:=\max\{\delta(\mathfrak{H},\mathfrak{H}_{h}),\delta(\mathfrak{H}_{h},\mathfrak{H})\}.
\end{aligned}
\end{equation*}
The next lemma is essentially a combination of results in \cite{HLMS2019,DH2014,AFW2010}. To be complete and precise, we give an explicit upper bound for the gap.
\begin{lemma}\label{gaph}
Let $\pi_{h}$ be a $V$-bounded cochain projection from $(V,d)$ to $(V_{h},d)$. It holds that
\begin{equation}\label{HHh}
\gap(\mathfrak{H},\mathfrak{H}_{h})=\delta(\mathfrak{H},\mathfrak{H}_{h})=\delta(\mathfrak{H}_{h},\mathfrak{H})\leq\left(1-\frac{1}{\big(\|\pi_{h}\|_{V}+1\big)^{2}}\right)^{\frac{1}{2}}.
\end{equation}
Let $(V_{H},d)$ be a subcomplex of $(V_{h},d)$ and $\pi_{H}$ be a $V$-bounded cochain projection from $(V,d)$ to $(V_{H},d)$. Then
\begin{equation}\label{HhHH}
\gap(\mathfrak{H}_{h},\mathfrak{H}_{H})=\delta(\mathfrak{H}_{h},\mathfrak{H}_{H})=\delta(\mathfrak{H}_{H},\mathfrak{H}_{h})\leq\left(1-\frac{1}{\big(\|\pi_{H}\|_{V}+1\big)^{2}}\right)^{\frac{1}{2}}.
\end{equation}
\end{lemma}
\begin{proof}
Lemma 2 in \cite{DH2014} gives $\delta(\mathfrak{H},\mathfrak{H}_{h})=\delta(\mathfrak{H}_{h},\mathfrak{H})$. For any $q\in\mathfrak{H}_{h}$, we have the estimate (see Theorem 3.5 in \cite{AFW2010})
\begin{equation}\label{estimateh}
\|q-P_{\mathfrak{H}}q\|\leq\|(I-\pi_{h})P_{\mathfrak{H}}q\|.
\end{equation}
Note that $\|\tilde{q}\|=\|\tilde{q}\|_{V}$ for any $\tilde{q}\in\mathfrak{H}$ or $\mathfrak{H}_{h}$. Combining \eqref{estimateh} with a triangle inequality gives $\|q\|\leq\|q-P_{\mathfrak{H}}q\|+\|P_{\mathfrak{H}}q\|\leq\big(\|I-\pi_{h}\|_{V}+1\big)\|P_{\mathfrak{H}}q\|.$ Since $\pi_{h}$ is a projection, we have $\|I-\pi_{h}\|_{V}=\|\pi_{h}\|_{V}$ (see \cite{XZ2003}) and thus $$\|q\|\leq\big(\|\pi_{h}\|_{V}+1\big)\|P_{\mathfrak{H}}q\|.$$ It then follows from the above inequality and $\|q-P_{\mathfrak{H}}q\|^{2}=\|q\|^{2}-\|P_{\mathfrak{H}}q\|^{2}$ that
\begin{equation*}
    \begin{aligned}
     \delta(\mathfrak{H}_{h},\mathfrak{H})&=\sup_{\|q\|=1,q\in\mathfrak{H}_{h}}\sqrt{1-\|P_{\mathfrak{H}}q\|^{2}}\\
     &\leq\left(1-\frac{1}{\big(\|\pi_{h}\|_{V}+1\big)^{2}}\right)^{\frac{1}{2}}.
    \end{aligned}
\end{equation*}

For any $q\in\mathfrak{H}_{H}\subset\mathfrak{Z}_{h}=\mathfrak{B}_{h}\oplus\mathfrak{H}_{h}$, $q-P_{\mathfrak{H}_{h}}q\in\mathfrak{B}_{h}$ and thus $\pi_{H}(q-P_{\mathfrak{H}_{h}}q)\in\mathfrak{B}_{H}.$ Hence $\pi_{H}(q-P_{\mathfrak{H}_{h}}q)\perp q-P_{\mathfrak{H}_{h}}q$, and 
\begin{equation}\label{estimateH}
\|q-P_{\mathfrak{H}_{h}}q\|\leq\|q-P_{\mathfrak{H}_{h}}q-\pi_{H}(q-P_{\mathfrak{H}_{h}}q)\|=\|(I-\pi_{H})P_{\mathfrak{H}_{h}}q\|.
\end{equation}
Repalcing \eqref{estimateh} by \eqref{estimateH} and following the same argument in the proof of \eqref{HHh}, we obtain \eqref{HhHH}.
\end{proof}

Christiansen and Winther \cite{CW2008} gave a $\pi_{h}$ that is uniformly bounded in the $L^{2}$-norm (thus in the $V$-norm) on shape regular meshes. Let $V_{H}^{k}\subset V_{h}^{k}$ be two nested finite element spaces generated by AMFEM. Then $\gap(\mathfrak{H}_{h}^{k},\mathfrak{H}_{H}^{k})$ and $\gap(\mathfrak{H}^{k},\mathfrak{H}_{h}^{k})$ are uniformly bounded away from $1$ provided the mesh refinement algorithm preserves shape regularity.

The major component of this section is a separate a posteriori error estimator for $\|\sigma-\sigma_{h}\|_{H\Lambda^{k-1}}$, which is crucial for proving convergence and optimality of our adaptive algorithms   w.r.t.~the error $\|\sigma-\sigma_{h}\|_{H\Lambda^{k-1}}$. 
\begin{theorem}[separate error estimator for $\|\sigma-\sigma_{h}\|_{H\Lambda^{k-1}}$]\label{bdsigma}
For $f\in H^{1}\Lambda^{k}(\Th)$ with $1\leq k\leq n-1$ or $f\in L^2\Lambda^{n}(\Omega)$, there exists a constant $C_{\up}$ depending solely on $\Omega$ and the shape regularity of $\Th$, such that
\begin{equation*}
\|\sigma-\sigma_{h}\|^{2}_{H\Lambda^{k-1}}\leq C_{\up}\eta^{2}_{\sigma}(\Th)=C_{\up}\sum_{K\in\Th}\eta^{2}_{\sigma}(K),
\end{equation*}
where
\begin{equation*}
\eta^{2}_{\sigma}(K)=\left\{\begin{aligned}&h_{K}^{2}\|\delta(f-d\sigma_{h})\|^{2}_{K}+h_{K}\|\lr{\tr\star(f-d\sigma_{h})}\|_{\partial K}^{2},\quad&&k=1,\\
&h_{K}^{2}\|\delta(f-d\sigma_{h})\|^{2}_{K}+h_{K}^{2}\|\delta\sigma_{h}\|_{K}^{2}\\
&+h_{K}\|\lr{\tr\star(f-d\sigma_{h}})\|_{\partial K}^{2}+h_{K}\|\lr{\tr\star\sigma_{h}}\|_{\partial K}^{2},&&2\leq k\leq n-1,\\
&h_K^2\|\delta\sigma_h\|_K^2+h_K\|\lr{\tr\star\sigma_h}\|_{\partial K}^2+\|f-f_{\Th}\|_K^2,&&k=n,
\end{aligned}\right.
\end{equation*}
$f_{\Th}$ is the $L^2$ projection of $f$ onto $V_h^n$.
\end{theorem}
\begin{proof}
Assume $2\leq k\leq n-1$. We use Theorem \ref{bddsigma} to estimate $\|d(\sigma-\sigma_{h})\|$. To estimate $\|\sigma-\sigma_{h}\|$, let $\sigma-\sigma_{h}=dv_{1}+\delta v_{2}+q$ be the Hodge decomposition of $\sigma-\sigma_{h}$, where $v_{1}\in\mathfrak{Z}^{\perp V}$, $v_{2}\in\mathfrak{Z}^{*\perp V^*}$, and $q\in\mathfrak{H}$. Our strategy is to estimate each component in the decomposition separately. In doing so, let $v_{1}=d\varphi_{1}+z_{1}$ be a regular decomposition of $v_{1}$ and recall that $\Pi_{h}$ is the commuting quasi-interpolation in Theorem \ref{rdqi}. Then $dv_{1}=dz_{1}$ and $d(v_{1}-\Pi_{h}v_{1})=d(z_{1}-\Pi_{h}z_{1})$. By Theorem \ref{rdqi} and a Poincar\'e inequality, 
\begin{equation}\label{bdzvdv}
\|z_1\|_{H^{1}}\lesssim\|v_{1}\|_{H\Lambda^{k-2}}\lesssim\|dv_{1}\|.
\end{equation}
Using $\sigma\in\mathfrak{Z}^{\perp}, \sigma_{h}\in\mathfrak{Z}_{h}^{\perp}$ and \eqref{IP} for element-wise integration by parts, we have
\begin{equation*}
\begin{aligned}
\|dv_{1}\|^{2}&=\ab{\sigma-\sigma_{h}}{dv_{1}}\\
        &=\ab{-\sigma_{h}}{d(v_{1}-\Pi_{h}v_{1})}\\
        &=\ab{-\sigma_{h}}{d(z_{1}-\Pi_{h}z_{1})}\\
        &=\sum_{K\in\Th}-\ab{\delta\sigma_{h}}{z_{1}-\Pi_{h}z_{1}}_{K}-\int_{\partial K}\tr\star\sigma_{h}\wedge\tr(z_{1}-\Pi_{h}z_{1}).
\end{aligned}
\end{equation*}
{Since $z_{1}\in H^1\Lambda^{k-2}(\Omega)$ and $\Pi_{h}z_{1}$ is a finite element form in $V_h^{k-2}$, the trace $\tr(z_{1}-\Pi_{h}z_{1})\in H^{\frac{1}{2}}\Lambda^{k-2}(e)$ is well-defined on the face $e\subset\partial K$ and thus $\tr(z_{1}-\Pi_{h}z_{1})$  has no jump across $e$.} It then follows from regrouping the sum over all $\partial K$ and the Cauchy--Schwarz inequality that
\begin{equation*}
\begin{aligned}
\|dv_{1}\|^{2} &=\sum_{K\in\Th}-\ab{\delta\sigma_{h}}{z_{1}-\Pi_{h}z_{1}}_{K}-\sum_{e\in\Eh}\int_{e}\lr{\tr\star\sigma_{h}}\wedge\tr(z_{1}-\Pi_{h}z_{1})\\
        &\lesssim\left(\sum_{K\in\Th}h_{K}^{2}\|\delta\sigma_{h}\|_{K}^{2}+h_{K}\|\llbracket\tr\star\sigma_{h}\rrbracket\|_{\partial K}^{2}\right)^{\frac{1}{2}}\\
        &\quad\times\left(\sum_{K\in\Th}h_{K}^{-2}\|z-\Pi_{h}z\|_{K}^{2}+h_{K}^{-1}\|\tr(z-\Pi_{h}z)\|^{2}_{\partial K}\right)^{\frac{1}{2}}.
\end{aligned}
\end{equation*}
Using the approximation property of $\Pi_{h}$ in Theorem \ref{rdqi} and the bounds \eqref{bdzvdv}, we obtain
\begin{equation}\label{dv1}
\|dv_{1}\|\lesssim\left(\sum_{K\in\Th}h_{K}^{2}\|\delta\sigma_{h}\|_{K}^{2}+h_{K}\|\llbracket\tr\star\sigma_{h}\rrbracket\|_{\partial K}^{2}\right)^{\frac{1}{2}}.
\end{equation}
   
For the component $\delta v_{2}$ with $v_{2}\in\mathfrak{Z}^{*\perp V^{*}}\subset\mathfrak{Z}^{*\perp}=\mathfrak{B}$, let $v_{2}=dw$ where $w\in\mathfrak{Z}^{\perp V}$. Let $w=d\varphi_{2}+z_{2}$ be a regular decomposition of $w$. By the error equation \eqref{error2} and $p\perp\mathfrak{B}\supset\mathfrak{B}_{h},\ p_{h}\perp\mathfrak{B}_{h}$, we have
\begin{equation*}
\begin{aligned}
        \|\delta v_{2}\|^{2}&=\ab{d(\sigma-\sigma_{h})}{dw}=\ab{d(\sigma-\sigma_{h})}{d(w-\Pi_{h}w)}.
\end{aligned}
\end{equation*}
It then follows from $d(w-\Pi_{h}w)=d(z_{2}-\Pi_{h}z_{2})$ and $d\sigma=f-\delta du-p$ that
\begin{equation}\label{interdeltav2}
\begin{aligned}
        \|\delta v_{2}\|^{2}
        &=\ab{f-p-\delta du-d\sigma_{h}}{d(z_{2}-\Pi_{h}z_{2})}\\
        &=\ab{f-d\sigma_{h}}{d(z_{2}-\Pi_{h}z_{2})}.
\end{aligned}
\end{equation}
Similarly using \eqref{interdeltav2}, element-wise integration by parts, Theorem \ref{rdqi}, and bounds 
$$\|z_{2}\|_{H^{1}}\lesssim\|w\|_{H\Lambda^{k-1}}\lesssim\|dw\|=\|v_{2}\|\lesssim\|\delta v_{2}\|,$$
we obtain
\begin{equation}\label{deltav2}
\begin{aligned}
        \|\delta v_{2}\|^{2}&=\sum_{K\in\Th}\ab{\delta(f-d\sigma_{h})}{z_{2}-\Pi_{h}z_{2}}_{K}\\
        &+\sum_{e\in\Eh}\int_{e}\lr{\tr\star(f-d\sigma_{h})}\wedge\tr(z_{2}-\Pi_{h}z_{2})\\
        &\lesssim\left(\sum_{K\in\mathcal{T}_{h}}h_{K}^{2}\|\delta(f-d\sigma_{h})\|_{K}^{2}+h_{K}\|\lr{\tr\star(f-d\sigma_{h})}\|_{\partial K}^{2}\right)^{\frac{1}{2}}\\
        &\quad\times\left(\sum_{K\in\Th}h_{K}^{-2}\|z_{2}-\Pi_{h}z_{2}\|_{K}^{2}+h_{K}^{-1}\|\tr(z_{2}-\Pi_{h}z_{2})\|^{2}_{\partial K}\right)^{\frac{1}{2}}\\
        &\lesssim\left(\sum_{K\in\Th}\eta^{2}_{d\sigma}(K)\right)^{\frac{1}{2}}\|w\|_{H\Lambda^{k-1}}\lesssim\sum_{K\in\Th}\eta^{2}_{d\sigma}(K).
\end{aligned}
\end{equation}

In the end, the harmonic component is estimated by
\begin{equation}\label{qs}
        \begin{aligned}
        \|q\|&=\ab{\sigma-\sigma_{h}}{\frac{q}{\|q\|}}\\
        &=\ab{\sigma-\sigma_{h}}{\frac{q}{\|q\|}-P_{\mathfrak{H}_{h}}\frac{q}{\|q\|}}\\
        &\leq\delta(\mathfrak{H},\mathfrak{H}_{h})\|\sigma-\sigma_{h}\|.
        \end{aligned}
\end{equation}
    
By \eqref{HHh} in Lemma \ref{gaph}, $\delta(\mathfrak{H},\mathfrak{H}_{h})\leq C(\Th,\Omega)<1$, where $C(\Th,\Omega)$ is a constant depending on $\Omega$ and the shape regularity of $\Th$. Combining the above three bounds \eqref{dv1}, \eqref{deltav2}, and \eqref{qs}, we have
\begin{equation*}
\begin{aligned}
        \|\sigma-\sigma_{h}\|^{2}&=\|dv_{1}\|^{2}+\|\delta v_{2}\|^{2}+\|q\|^{2}\\
        &\leq\frac{1}{1-C(\Th,\Omega)^{2}}\big(\|dv_{1}\|^{2}+\|\delta v_{2}\|^{2}\big)\\
        &\lesssim\sum_{K\in\mathcal{T}_{h}}h_{K}^{2}\|\delta\sigma_{h}\|_{K}^{2}+h_{K}\|\lr{\tr\star\sigma_{h}}\|_{\partial K}^{2}+\eta^{2}_{d\sigma}(K).
\end{aligned}
\end{equation*}
When $k=1$, $\sigma-\sigma_{h}\in H\Lambda^{0}(\Omega)$ and then there is no boundary component $dv_{1}$ in the Hodge decomposition. In this case, $\|\sigma-\sigma_{h}\|^{2}\lesssim\sum_{K\in\Th}\eta^{2}_{d\sigma}(K)$. When $k=n$, $d\sigma=f$, $d\sigma_h=f_{\Th}$, and we replace $\eta_{d\sigma}(\Th)$ by $\|f-f_{\Th}\|$.
The proof is complete.
\end{proof}
 
We compare $\eta_{\sigma}(\Th)$ with the a posteriori error estimator  
\begin{equation}\label{errorDH}
\eta_{\text{DH}}(\Th)=\left(\sum_{K\in\Th}\eta^{2}_{-1}(K)+\eta^{2}_{0}(K)+\eta^{2}_{\mathfrak{H}}(K,p_{h})\right)^{\frac{1}{2}}+\mu\|u_{h}\|,
\end{equation}
in \cite{DH2014}, where $\mu$ is some assumed a posteriori upper bound on $\gap(\mathfrak{H},\mathfrak{H}_{h})$. Practical a posteriori estimates for $\gap(\mathfrak{H},\mathfrak{H}_{h})$ can be found in \cite{Demlow2017}. However, there is no efficiency result on the term $\mu\|u_{h}\|$, i.e., $\mu\|u_{h}\|\lesssim\|\sigma-\sigma_{h}\|_{H\Lambda^{k-1}}+\|u-u_{h}\|_{H\Lambda^{k}}+\|p-p_{h}\|$. In contrast, a local efficiency result of $\eta_{\sigma}(K)$ can be directly established by using the Verf\"urth bubble function technique (see \cite{DH2014} and \cite{V1996}). 
For $K\in\Th$, let $\omega_{K}$ denote the collection of $K$ and neighboring simplices in $\Th$ sharing an $(n-1)$-face with $K$. Let $\Omega_K=\cup_{K^\prime\in\omega_K}K^\prime$ denote the local patch surrounding $K.$
\begin{theorem}[efficiency]\label{eff}
For $K\in\Th,$ let $\osc_{\Th}(\sigma_{h},f,K)=0$ when $k=n$ and
\begin{equation*}
\begin{aligned}
&\osc^{2}_{\Th}(\sigma_{h},f,K)=
h_{K}^{2}\|(\emph{id}-Q_{K})\delta(f-d\sigma_{h})\|_{K}^{2}\\
&\qquad+h_{K}\|(\emph{id}-Q_{\partial K})\lr{\tr\star(f-d\sigma_{h})}\|_{\partial K}^{2}\text{ when   } 1\leq k\leq n-1,
\end{aligned}
\end{equation*}
where $Q_{K}$ is the $L^{2}$-projection onto the space of polynomial $(k-1)$-forms of degree $\leq r$ on $K$, $Q_{\partial K}$ is the $L^{2}(\partial K)$-projection onto the space of discontinuous piecewise polynomial $(n-k)$-forms of degree $\leq r^{\prime}$ on $\partial K$. 
For $f\in H^{1}\Lambda^{k}(\Th)$ with $1\leq k\leq n-1$ or $f\in L^2\Lambda^n(\Omega)$, there exists a constant $C_{\low}$ depending solely on $r, r^\prime, \Omega$ and the shape regularity of $\Th$, such that
\begin{equation*}
\begin{aligned}
C_{\low}\eta^{2}_{\sigma}(K)&\leq\|\sigma-\sigma_{h}\|^{2}_{H\Lambda^{k-1}(\Omega_{K})}+\sum_{K^\prime\in\omega_{K}}\osc^{2}_{\Th}(\sigma_{h},f,K^\prime).
\end{aligned}
\end{equation*}
\end{theorem}

Note that $\osc_{\Th}(\sigma_{h},f,K)$ here is slightly different from the data oscillation given by (5.13)--(5.15) in \cite{DH2014}. The advantage here is the dominance in Theorem \ref{perturb},
which is helpful for proving optimality.
    
We then give an a posteriori estimate for the harmoinc error.
\begin{theorem}[separate estimator for $\|p-p_{h}\|$]\label{bdp}
For $1\leq k\leq n-1$ and $f\in H^{1}\Lambda^{k}(\Th)$, it holds that 
\begin{equation*}
\|p-p_{h}\|^2\lesssim\eta^2_{p}(\Th)=\sum_{K\in\Th}\eta^{2}_{p}(K),
\end{equation*}
where
\begin{equation*}
\eta_{p}^{2}(K)=h_{K}^{2}\|\delta p_{h}\|_{K}^{2}+h_{K}\|\lr{\tr\star p_{h}}\|_{\partial K}^{2}+\eta^{2}_{d\sigma}(K).
\end{equation*}
\end{theorem}
\begin{proof}
Let $p-p_{h}=(p-P_{\mathfrak{H}}p_{h})+(P_{\mathfrak{H}}p_{h}-p_{h}).$ Since $p_{h}\in\mathfrak{Z}_{h}\subset\mathfrak{Z}=\mathfrak{B}\oplus\mathfrak{H}$,  $p_{h}-P_{\mathfrak{H}}p_{h}\in\mathfrak{B}$. Therefore $p_{h}-P_{\mathfrak{H}}p_{h}\perp q$ and 
\begin{equation}\label{totalp}
\begin{aligned}
        \|p-P_{\mathfrak{H}}p_{h}\|&=\sup_{q\in\mathfrak{H},\|q\|=1}\ab{p-P_{\mathfrak{H}}p_{h}}{q}=\sup_{q\in\mathfrak{H},\|q\|=1}\ab{p-p_{h}}{q}\\
        &=\sup_{q\in\mathfrak{H},\|q\|=1}\big(\ab{p-p_{h}}{q-P_{\mathfrak{H}_{h}}q}+\ab{p-p_{h}}{P_{\mathfrak{H}_{h}}q}\big)\\
        &\leq\delta(\mathfrak{H},\mathfrak{H}_{h})\|p-p_{h}\|+\sup_{q\in\mathfrak{H},\|q\|=1}\ab{p-p_{h}}{P_{\mathfrak{H}_{h}}q}.
\end{aligned}
\end{equation}
The error equation \eqref{error2} implies
\begin{equation}\label{bdp1}
\begin{aligned}
\sup_{q\in\mathfrak{H},\|q\|=1}\ab{p-p_{h}}{P_{\mathfrak{H}_{h}}q}&=\sup_{q\in\mathfrak{H},\|q\|=1}\ab{d(\sigma-\sigma_{h})}{-P_{\mathfrak{H}_{h}}q}\\
&=\sup_{q\in\mathfrak{H},\|q\|=1}\ab{d(\sigma-\sigma_{h})}{q-P_{\mathfrak{H}_{h}}q}\\
        &\leq\delta(\mathfrak{H},\mathfrak{H}_{h})\|d(\sigma-\sigma_{h})\|.
\end{aligned}
\end{equation}

By the equation (2.12) and Lemma 9 in \cite{DH2014}, 
\begin{equation}\label{bdp2}
        \begin{aligned}
        \|P_{\mathfrak{H}}p_{h}-p_{h}\|&\lesssim\sup_{\|\phi\|_{H\Lambda^{k-1}}=1}\ab{p_{h}}{d(\phi-\Pi_{h}\phi)}\\
        &\lesssim\left(\sum_{K\in\mathcal{T}_{h}}h_{K}^{2}\|\delta p_{h}\|_{K}^{2}+h_{K}\|\lr{\tr\star p_{h}}\|_{\partial K}^{2}\right)^{\frac{1}{2}}.
        \end{aligned}
\end{equation}

Combining \eqref{totalp}--\eqref{bdp2} and using Theorem \ref{bddsigma}, we obtain
\begin{equation*}
\begin{aligned}
\|p-p_{h}\|&\leq\|p-P_{\mathfrak{H}}p_{h}\|+\|P_{\mathfrak{H}}p_{h}-p_{h}\|\\
&\leq\frac{1}{1-\delta(\mathfrak{H},\mathfrak{H}_{h})}\big(\delta(\mathfrak{H},\mathfrak{H}_{h})\|d(\sigma-\sigma_{h})\|+\|P_{\mathfrak{H}}p_{h}-p_{h}\|\big)\\
&\lesssim\left(\sum_{K\in\mathcal{T}_{h}}h_{K}^{2}\|\delta p_{h}\|_{K}^{2}+h_{K}\|\lr{\tr\star p_{h}}\|_{\partial K}^{2}\right)^{\frac{1}{2}}+\eta_{d\sigma}(\Th).
\end{aligned}
\end{equation*}
The proof is complete.
\end{proof}
When $k=n$, $\mathfrak{H}_h^n=\mathfrak{H}^n=\{0\}$, and $\eta_p$ is useless. The efficiency of $\eta_{p}$ follows from the efficiency argument in \cite{DH2014}.
\begin{theorem}[efficiency of $\eta_{p}(K)$]
For $1\leq k\leq n-1$ and $f\in H^{1}\Lambda^{k}(\Th)$, the local efficiency holds:
\begin{equation*}
\begin{aligned}
\eta^2_{p}(K)&\lesssim\|p-p_{h}\|^2_{\Omega_{K}}+\|d(\sigma-\sigma_h)\|^2_{\Omega_K}\\
&\quad+\sum_{K^\prime\in\omega_K}h^2_{K^\prime}\|\delta(f-P_{h}f)\|^2_{K^\prime}+h_{K^\prime}\|\lr{\tr\star(f-P_{h}f)}\|^2_{\partial K^\prime},
\end{aligned}
\end{equation*}
where $P_{h}f$ is the $L^2$-projection of $f$ onto the space of $k$-forms with discontinuous piecewise polynomial coefficients of arbitrary but fixed degree.
\end{theorem}

The next theorem gives a posteriori estimate on $\|d(u-u_{h})\|$, which is similar to the error estimator in \cite{CW2017} but involves the harmonic term $p_h$ here. 
\begin{theorem}[separate estimator for $\|d(u-u_{h})\|$]\label{bddu}
For $1\leq k\leq n-1$ and $f\in H^{1}\Lambda^{k}(\Th)$, we have the a posteriori estimate
\begin{equation*}
\|d(u-u_{h})\|^2\lesssim\eta_{du}^2(\Th)=\sum_{K\in\Th}\eta^{2}_{du}(K),
\end{equation*}
where
\begin{equation*}
\begin{aligned}
\eta^{2}_{du}(K)&=h_{K}^{2}\|f-d\sigma_{h}-\delta du_{h}-p_{h}\|_{K}^{2}+h_{K}^{2}\|\delta(f-d\sigma_{h}-p_{h})\|_{K}^{2}\\&+h_{K}\|\lr{\tr\star(f-d\sigma_{h}-p_{h}}\|_{\partial K}^{2}+h_{K}\|\lr{\tr\star du_{h}}\|_{\partial K}^{2}.
\end{aligned}
\end{equation*}
\end{theorem}
\begin{proof}
Let $v\in\mathfrak{Z}^{\perp V}$ such that $dv=d(u-u_{h})$. Let $v=d\varphi+z$ be the regular decomposition of $v$. $\eqref{error2}$ implies 
\begin{equation*}
\begin{aligned}
        \|d(u-u_{h})\|^{2}&=\ab{d(u-u_{h})}{dv}\\
        &=\ab{d(u-u_{h})}{d(v-\Pi_{h}v)}-\ab{d(\sigma-\sigma_{h})}{\Pi_{h}v}-\ab{\Pi_{h}v}{p-p_{h}}.
\end{aligned}
\end{equation*}
Then by $v\perp\mathfrak{Z}^k$, $f=d\sigma+\delta du+p$, and $v-\Pi_{h}v=d(\varphi-\Pi_{h}\varphi)+(z-\Pi_{h}z)$, we have
\begin{equation*}
\begin{aligned}
        \|d(u-u_{h})\|^{2}&=\ab{d(u-u_{h})}{d(v-\Pi_{h}v)}+\ab{d(\sigma-\sigma_{h})}{v-\Pi_{h}v}\\
        &\qquad+\ab{v-\Pi_{h}v}{p-p_{h}}\\
        &=\ab{f-d\sigma_{h}-p_{h}}{v-\Pi_{h}v}-\ab{du_{h}}{d(v-\Pi_{h}v)}\\
        &=\ab{f-d\sigma_{h}-p_{h}}{d(\varphi-\Pi_{h}\varphi)+z-\Pi_{h}z}-\ab{du_{h}}{d(z-\Pi_{h}z)}.
\end{aligned}
\end{equation*}
The theorem then follows from the standard element-wise integration  by parts, Theorem \ref{rdqi}, and the Poincar\'e inequality $\|v\|_{H\Lambda^{k}}\lesssim\|dv\|$.
\end{proof}
In fact, $\eta_{du}(K)$ in Theorem \ref{bddu} is just  $\eta_{0}(K)$ in \cite{DH2014}. Local efficiency of $\eta_{du}$ can be found from Lemma 12 in \cite{DH2014}.

In the end, we can proceed to derive the a posteriori error estimate for $\|u-u_{h}\|$ by following the proof of Theorem \ref{bdsigma}. However, the corresponding error bound would be similar to the coupled error bound $\eta_{\text{DH}}$ \eqref{errorDH}. In particular, it involves the term $\|P_{\mathfrak{H}}u_{h}\|\leq\mu\|u_{h}\|$. Therefore the separate error estimator for $\|u-u_{h}\|$ has no advantage.

\section{Convergence}\label{convergence}
Given a subset $\mathcal{M}\subset\mathcal{T}_{h},$ define
$$\eta^2_{\sigma}(\mathcal{M}):=\sum_{K\in\mathcal{M}}\eta^{2}_{\sigma}(K),$$
and similar for $\eta_{p}(\mathcal{M})$ and $\eta_{du}(\mathcal{M})$.  
We now present the adaptive algorithm \textsf{AMFEM1} for solving the Hodge Laplacian problem \eqref{HL} based on the standard adaptive feedback loop
$$\textsf{SOLVE}\longrightarrow\textsf{ESTIMATE}\longrightarrow\textsf{MARK}\longrightarrow\textsf{REFINE}.$$ 
\textsf{AMFEM1} is designed to reduce the error $\|\sigma-\sigma_{h}\|_{H\Lambda^{k-1}}^2+\|p-p_{h}\|^2+\|d(u-u_{h})\|^2$. It is unconditionally convergent starting from any initial coarse mesh. 
\begin{algorithm}
\textsf{AMFEM1}. Given an initial mesh $\mathcal{T}_{0}$, marking parameters $0<\theta_{\sigma}, \theta_{p}, \theta_{du}<1$, and an error tolerance $\tol>0$. Set $\ell=0$.
\begin{itemize}
\item[]\textbf{Step 1} Solve \eqref{DHL} on $\mathcal{T}_{\ell}$ to obtain the finite element solution $(\sigma_{\ell},u_{\ell},p_{\ell})$. 
\item[]\textbf{Step 2} Compute the error indicators $\eta_{\sigma}(K)$, $\eta_{p}(K)$ and $\eta_{du}(K)$ on each element $K\in\mathcal{T}_{\ell}$ and $\eta_{\ell}=\big(\eta^{2}_{\sigma}(\mathcal{T}_{\ell})+\eta^{2}_{p}(\mathcal{T}_{\ell})+\eta^{2}_{du}(\mathcal{T}_{\ell})\big)^{\frac{1}{2}}.$ If $\eta_\ell\leq\tol$, return $(\sigma_\ell,u_\ell,p_\ell)$ and $\mathcal{T}_\ell$; else go to \textbf{Step 3}.
\item[]\textbf{Step 3} Select a subset $\mathcal{M}_{\ell}$ of $\mathcal{T}_{\ell}$ such that \begin{align}
\eta_{\sigma}(\mathcal{M}_{\ell})&\geq\theta_{\sigma}\eta_{\sigma}(\mathcal{T}_{\ell}),\label{mark1}\\
\eta_{p}(\mathcal{M}_{\ell})&\geq\theta_{p}\eta_{p}(\mathcal{T}_{\ell}),\label{mark2}\\
\eta_{du}(\mathcal{M}_{\ell})&\geq\theta_{du}\eta_{du}(\mathcal{T}_{\ell}).\label{mark3}
\end{align}
\item[]\textbf{Step 4} Refine each element in $\mathcal{M}_{\ell}$ and necessary neighboring elements by a mesh refinement algorithm preserving shape regularity to get a conforming mesh $\mathcal{T}_{\ell+1}$. Set $\ell=\ell+1$ and go to \textbf{Step 1}.
\end{itemize}
\end{algorithm}

{As in Theorems \ref{bddsigma}, \ref{bdp}, and \ref{bddu}, $f$ is required to be contained in $H^1\Lambda^k(\mathcal{T}_0)$ when $1\leq k\leq n-1$, that is, the discontinuity of $f$ is aligned with the initial mesh $\mathcal{T}_0$ in the adaptive algorithm. }

\textbf{Step 3} is often called D\"orfler marking in the literature. Marking properties \eqref{mark1}--\eqref{mark3} can be achieved by first selecting $\mathcal{M}_\ell$ such that \eqref{mark1} holds, then successively enlarging $\mathcal{M}_\ell$ to make \eqref{mark2} and \eqref{mark3} satisfied. The marking step can be flexible, see Remark \ref{remarkAMFEM1} for details. Candidates for mesh refinement in \textbf{Step 4} include bisection or quad-refinement with bisection closure, see e.g., \cite{Mitchell1990,Traxler1997,Bey2000,V1996}. 
 
Let $\{(\sigma_{\ell},p_{\ell},u_{\ell}),\mathcal{T}_{\ell}\}_{\ell\geq0}$ be a sequence of finite element solutions and meshes produced by \textsf{AMFEM1}. Let 
\begin{equation*}
\begin{aligned}
e_{d\sigma,\ell}&=\|d(\sigma-\sigma_{\ell})\|^{2},\quad \eta_{d\sigma,\ell}=\eta^{2}_{d\sigma}(\mathcal{T}_{\ell}),\quad \Delta_{d\sigma,\ell}=\|d(\sigma_{\ell}-\sigma_{\ell+1})\|^{2},\\
e_{\sigma,\ell}&=\|\sigma-\sigma_{\ell}\|^{2},\quad \eta_{\sigma,\ell}=\eta^{2}_{\sigma}(\mathcal{T}_{\ell}),\quad\Delta_{\sigma,\ell}=\|\sigma_{\ell}-\sigma_{\ell+1}\|^{2},
\end{aligned}
\end{equation*}
and similar for other quantities. We list ingredients for proving convergence of \textsf{AMFEM1} in the next two lemmas.

\begin{lemma}\label{lemma1}
For $\ell\geq0$,
\begin{equation}\label{ortho1}
e_{d\sigma,\ell+1}=e_{d\sigma,\ell}-\Delta_{d\sigma,\ell}.
\end{equation}
In addition, for arbitrary $\{\varepsilon_{i}\}_{i=1}^{3}\subset(0,1),$ there exist constants $C^\sigma_{\emph{qo}}, C^{du}_{\emph{qo}}>0$ depending only on $\mathcal{T}_0$ and $\Omega,$ such that
\begin{subequations}
\begin{align}
e_{\sigma,\ell+1}&\leq\frac{1}{1-\varepsilon_1}e_{\sigma,\ell}-\Delta_{\sigma,\ell}+\frac{\varepsilon_1^{-1}}{1-\varepsilon_1}C^\sigma_{\emph{qo}}\Delta_{d\sigma,\ell},\label{ortho2}\\
e_{p,\ell+1}&\leq e_{p,\ell}-(1-\varepsilon_2)\Delta_{p,\ell}+\varepsilon_2^{-1}e_{d\sigma,\ell+1},\label{ortho3}\\
e_{du,\ell+1}&\leq e_{du,\ell}-(1-\varepsilon_3)\Delta_{du,\ell}+2\varepsilon_3^{-1}C^{du}_{\emph{qo}}(e_{d\sigma,\ell+1}+e_{p,\ell+1})\label{ortho4}.
\end{align}
\end{subequations}
\end{lemma}
\begin{proof}
The orthogonality \eqref{ortho1} directly follows from \eqref{error2}. Let $\mathfrak{Z}_{\ell+1}^{\perp}$ be $\mathfrak{Z}_{h}^{\perp}$ based on $\mathcal{T}_{\ell+1}$. Using $\sigma-\sigma_{\ell+1}\perp\mathfrak{Z}_{\ell+1}$ and the discrete Poincar\'e inequality
$\|P_{\mathfrak{Z}_{\ell+1}^{\perp}}(\sigma_{\ell}-\sigma_{\ell+1})\|\leq (C^\sigma_{\text{qo}})^\frac{1}{2}\|dP_{\mathfrak{Z}_{\ell+1}^{\perp}}(\sigma_{\ell}-\sigma_{\ell+1})\|=(C_{\text{qo}}^\sigma\Delta_{d\sigma,\ell})^\frac{1}{2},$
we obtain
\begin{equation*}\label{orthosigma}
\begin{aligned}
e_{\sigma,\ell+1}&=e_{\sigma,\ell}-\Delta_{\sigma,\ell}+2\ab{\sigma-\sigma_{\ell+1}}{P_{\mathfrak{Z}_{\ell+1}^{\perp}}(\sigma_{\ell}-\sigma_{\ell+1})}\\
&\leq e_{\sigma,\ell}-\Delta_{\sigma,\ell}+2e_{\sigma,\ell+1}^\frac{1}{2}(C_{\text{qo}}^\sigma\Delta_{d\sigma,\ell})^\frac{1}{2},\\
&\leq e_{\sigma,\ell}-\Delta_{\sigma,\ell}+\varepsilon_1e_{\sigma,\ell+1}+\varepsilon_1^{-1}C_{\text{qo}}^\sigma\Delta_{d\sigma,\ell}.
\end{aligned}
\end{equation*}
The proof of \eqref{ortho2} is complete. Similarly \eqref{error2} implies 
\begin{equation*}
\begin{aligned}
e_{p,\ell+1}&=e_{p,\ell}-\Delta_{p,\ell}+2\ab{p-p_{\ell+1}}{p_{\ell}-p_{\ell+1}}\\
&=e_{p,\ell}-\Delta_{p,\ell}-2\ab{d(\sigma-\sigma_{\ell+1})}{p_{\ell}-p_{\ell+1}}\\
&\leq e_{p,\ell}-\Delta_{p,\ell}+2e_{d\sigma,\ell+1}^\frac{1}{2}\Delta_{p,\ell}^\frac{1}{2}\\
&\leq e_{p,\ell}-(1-\varepsilon_2)\Delta_{p,\ell}+\varepsilon_2^{-1}e_{d\sigma,\ell+1}.
\end{aligned}
\end{equation*}
In the end, let $v_{\ell+1}\in\mathfrak{Z}^{\perp}_{\ell+1}$ such that $dv_{\ell+1}=d(u_{\ell}-u_{\ell+1})$. Using \eqref{error2} and $\|v_{\ell+1}\|\leq(C_{\text{qo}}^{du})^\frac{1}{2}\|dv_{\ell+1}\|=(C_{\text{qo}}^{du}\Delta_{du,\ell})^{\frac{1}{2}}$, \eqref{ortho3} can be proved by
\begin{equation*}
\begin{aligned}
e_{du,\ell+1}&=e_{du,\ell}-\Delta_{du,\ell}+2\ab{d(u-u_{\ell+1})}{d(u_{\ell}-u_{\ell+1})}\\
&=e_{du,\ell}-\Delta_{du,\ell}-2\ab{d(\sigma-\sigma_{\ell+1})}{v_{\ell+1}}-2\ab{v_{\ell+1}}{p-p_{\ell+1}}\\
&\leq e_{du,\ell}-\Delta_{du,\ell}+2(C_{\text{qo}}^{du})^\frac{1}{2}\big(e_{d\sigma,\ell+1}^{\frac{1}{2}}+e_{p,\ell+1}^{\frac{1}{2}})\Delta_{du,\ell}^{\frac{1}{2}},\\
&\leq e_{du,\ell}-\Delta_{du,\ell}+\varepsilon_3\Delta_{du,\ell}+2\varepsilon_3^{-1}C_{\text{qo}}^{du}\big(e_{d\sigma,\ell+1}+e_{p,\ell+1}).
\end{aligned}
\end{equation*}
The proof is complete.
\end{proof}
Lemma \ref{lemma1} deals with error reduction on two consecutive meshes. Another ingredient of convergence analysis is the following estimator reduction lemma.
\begin{lemma}\label{lemma2}
\begin{subequations}
\begin{align}
\eta_{\sigma,\ell+1}&\leq\beta_{\sigma}\eta_{\sigma,\ell}+C_{\sigma}(\Delta_{d\sigma,\ell}+\Delta_{\sigma,\ell}),\label{cts2}\\
\eta_{p,\ell+1}&\leq\beta_{p}\eta_{p,\ell}+C_{p}(\Delta_{d\sigma,\ell}+\Delta_{p,\ell}),\label{cts3}\\
\eta_{du,\ell+1}&\leq\beta_{du}\eta_{du,\ell}+C_{du}(\Delta_{d\sigma,\ell}+\Delta_{p,\ell}+\Delta_{du,\ell})\label{cts4}.
\end{align}
\end{subequations}
where $0<\beta_{\sigma}, \beta_{p}, \beta_{du}<1$ and $C_{\sigma}, C_p, C_{du}>0$ depend only on marking parameters $\theta_{\sigma}, \theta_{p}, \theta_{du}$ and $\mathcal{T}_{0}$. 
\end{lemma}
\begin{proof}
We only prove \eqref{cts2} since proofs of other inequalities are the same. Recall that $\eta_{\sigma,\ell}=\eta^{2}_{\sigma}(\mathcal{T}_{\ell})$ in Theorem \ref{bdsigma} depends only on $\sigma_\ell$ and $d\sigma_\ell$. 
Using the same argument in the proof of Corollary 3.4 in \cite{CKNS2008}, for arbitrary $\delta_{*}>0$, we have
\begin{equation}\label{ctseta}\eta_{\sigma,\ell+1}\leq(1+\delta_{*})\big(\eta_{\sigma,l}-\lambda\eta^{2}_{\sigma}(\mathcal{M}_{\ell})\big)+(1+\delta_{*}^{-1})C_{\mathcal{T}_0}(\Delta_{d\sigma,l}+\Delta_{\sigma,l}),
\end{equation}
where $\lambda=1-2^{-\frac{b}{n}}<1$, $b>0$ is an integer depending on mesh refinement strategy. Then \eqref{cts2} follows from \eqref{ctseta} and the marking property \eqref{mark2} $\eta^{2}_{\sigma,\ell}(\mathcal{M}_{\ell})\geq\theta_{\sigma}^{2}\eta_{\sigma,\ell}$ with $\beta_{\sigma}=(1+\delta_{*})(1-\lambda\theta_{\sigma}^{2}).$ $\beta_{\sigma}<1$ holds provided $\delta_{*}<\frac{\lambda\theta_{\sigma}^{2}}{1-\lambda\theta_{\sigma}^{2}}$.
\end{proof}

{Following the strategy in \cite{CKNS2008}, our convergence proof of \textsf{AMFEM1} is based on reliability results in Section \ref{aposteriori}, Lemmas \ref{lemma1} and \ref{lemma2}, and quasi-errors which are weighted sums of finite element errors and estimators.}
\begin{theorem}[convergence of \textsf{AMFEM1}]\label{convergenceAMFEM1} For $f\in H^1\Lambda^{k}(\mathcal{T}_0)$ with $1\leq k\leq n-1$, let $\{(\sigma_\ell,u_\ell,p_\ell),\mathcal{T}_\ell\}_{\ell\geq0}$ be a sequence of finite element solutions and meshes generated by Algorithm \textsf{AMFEM1}. Then there exist $\zeta, \rho_{\sigma}, \rho_p, \rho_{du}, C_{0}>0$, and $\gamma\in(0,1)$, depending only on $\Omega$, $\mathcal{T}_{0}$ and $\theta_{\sigma}, \theta_p, \theta_{du}$, such that
 \begin{equation*}
     \begin{aligned}
     &\|\sigma-\sigma_\ell\|^{2}+\zeta\|d(\sigma-\sigma_h)\|^2+\|p-p_\ell\|^{2}+\|d(u-u_\ell)\|^{2}\\
     &\quad+\rho_{\sigma}\eta^{2}_{\sigma}(\mathcal{T}_\ell)+\rho_p\eta^{2}_{p}(\mathcal{T}_\ell)+\rho_{du}\eta^{2}_{du}(\mathcal{T}_\ell)\leq C_{0}\gamma^\ell.
     \end{aligned}
 \end{equation*}
\end{theorem}
\begin{proof}
For convenience, we may use $C$ as a generic constant,   depending only on $\Omega$, $\mathcal{T}_{0}$ and possibly $\theta_{\sigma}, \theta_p, \theta_{du}$. Let $\rho_\sigma=1/C_{\sigma}$, and $E_\ell=e_{d\sigma,\ell}+e_{\sigma,\ell}+\rho_\sigma\eta_{\sigma,\ell}$. By \eqref{ortho1}, \eqref{ortho2} and \eqref{cts2}, 
\begin{equation*}
\begin{aligned}
E_{\ell+1}&\leq e_{d\sigma,\ell}-\Delta_{d\sigma,\ell}+\frac{1}{1-\varepsilon_1}e_{\sigma,\ell}+\rho_\sigma\beta_{\sigma}\eta_{\sigma,\ell}+\left(1+\frac{\varepsilon_1^{-1}}{1-\varepsilon_1}C_{\text{qo}}^\sigma\right)\Delta_{d\sigma,\ell}\\
        &\leq\frac{1}{1-\varepsilon_1}(e_{d\sigma,\ell}+e_{\sigma,\ell})+\rho_\sigma\beta_{\sigma}\eta_{\sigma,\ell}+C_{\varepsilon_1}\Delta_{d\sigma,\ell},
\end{aligned}
\end{equation*}
where $C_{\varepsilon_1}=C_{\text{qo}}^\sigma\varepsilon_1^{-1}/(1-\varepsilon_1)>1$. Let $\alpha_1\in(0,1)$ be an undetermined constant. Then by Theorem \ref{bdsigma}, 
\begin{equation}\label{reduction}
\begin{aligned}
E_{\ell+1}&\leq\alpha_1(e_{d\sigma,l}+e_{\sigma,\ell})\\
&+\left\{\left(\frac{1}{1-\varepsilon_1}-\alpha_1\right)C_{\up}+\rho_\sigma\beta_{\sigma}\right\}\eta_{\sigma,\ell}+C_{\varepsilon_1}\Delta_{d\sigma,\ell}.
\end{aligned}
\end{equation}
Solving the equation 
$$\left(\frac{1}{1-\varepsilon_1}-\alpha_1\right)C_{\up}+\rho_\sigma\beta_{\sigma}=\alpha_1\rho_\sigma$$
for $\alpha_1$, we obtain
\begin{equation*}
    \alpha_1=\frac{\frac{C_{\up}}{1-\varepsilon_1}+\rho_\sigma\beta_{\sigma}}{C_{\up}+\rho_\sigma}.
\end{equation*}
$\alpha_1<1$ provided \begin{equation*}
0<\varepsilon_1<\frac{1-\beta_{\sigma}}{C_{\sigma}C_{\up}+1-\beta_{\sigma}}
 \end{equation*}
 holds. In this case, \eqref{reduction} reduces to 
\begin{equation}\label{reduction2}
E_{\ell+1}\leq\alpha_1 E_{\ell}+C_{\varepsilon_1}\Delta_{d\sigma,\ell}.
\end{equation}
Let $\alpha_2\in(0,1)$ be a constant to be determined. Combining \eqref{reduction2}, \eqref{ortho1}, and $e_{d\sigma,\ell}\leq E_{\ell}$, we have
\begin{equation*}
\begin{aligned}
E_{\ell+1}+C_{\varepsilon_1}e_{d\sigma,\ell+1}&\leq\alpha_1 E_{\ell}+C_{\varepsilon_1}e_{d\sigma,\ell}\\
&\leq(\alpha_1+C_{\varepsilon_1}\alpha_2)E_{\ell}+C_{\varepsilon_1}(1-\alpha_2)e_{d\sigma,\ell}\\
&=\gamma_1\left\{E_{\ell}+\frac{C_{\varepsilon_1}(1-\alpha_2)}{\alpha_1+C_{\varepsilon_1}\alpha_2}e_{d\sigma,\ell}\right\},
\end{aligned}
\end{equation*}
where $\gamma_1=\alpha_1+C_{\varepsilon_1}\alpha_2$. We require that 
$$\gamma_1<1,\quad\frac{1-\alpha_2}{\alpha_1+C_{\varepsilon_1}\alpha_2}\leq1,$$
which is equivalent to 
$$0<\frac{1-\alpha_1}{1+C_{\varepsilon_1}}\leq\alpha_2<\frac{1-\alpha_1}{C_{\varepsilon_1}}<1.$$
By taking any $\alpha_2$ satisfying the above criterion, we proved  
\begin{equation}\label{reductionsigma}
\begin{aligned}
E_{\ell+1}+C_{\varepsilon_1}e_{d\sigma,\ell+1}\leq\gamma_1(E_{\ell}+C_{\varepsilon_1}e_{d\sigma,\ell})\leq C\gamma_1^{\ell+1},
\end{aligned}
\end{equation}
which gives contraction on $e_{\sigma,\ell}, e_{d\sigma,\ell}$, and $\eta_{\sigma,\ell}$.

Similarly, using \eqref{ortho3}, \eqref{cts3}, and Theorems \ref{bdp}, 
we have
\begin{equation}\label{con1}
\begin{aligned}
e_{p,\ell+1}+\rho_p\eta_{p,\ell+1}&\leq\alpha_{3}(e_{p,\ell}+\rho_p\eta_{p,\ell})+\rho_p C_p\Delta_{d\sigma,\ell}+\varepsilon_2^{-1}e_{d\sigma,\ell+1}\\
&\leq\alpha_3(e_{p,\ell}+\rho_p\eta_{p,\ell})+C(\eta_{\sigma,\ell}+\eta_{\sigma,\ell+1}),
\end{aligned}
\end{equation}
where
$\rho_p=(1-\varepsilon_2)/C_p$, $\alpha_{3}=\frac{C^p_{\up}+\rho_p\beta_p}{C^p_{\up}+\rho_p}<1$, and
$C^p_{\up}$ is the multiplicative constant in Theorem \ref{bdp}. Let $\tilde{\gamma}_{2}=\max(\alpha_3,\gamma_1)$ It then follows from \eqref{con1} and \eqref{reductionsigma} that
\begin{equation*}
\begin{aligned}
e_{p,\ell+1}+\rho_p\eta_{p,\ell+1}\leq\tilde{\gamma}_2(e_{p,\ell}+\rho_p\eta_{p,\ell})+C\tilde{\gamma}_2^\ell.
\end{aligned}
\end{equation*}
By induction and taking $\gamma_2=\sqrt{\tilde{\gamma}_2}$, we obtain
\begin{equation}\label{reductionp}
\begin{aligned}
e_{p,\ell}+\rho_p\eta_{p,\ell}&\leq\tilde{\gamma}_2^\ell(e_{p,0}+\rho_p\eta_{p,0})+C\ell\tilde{\gamma}_2^\ell\leq C\gamma_2^\ell.
\end{aligned}
\end{equation}

In the end, it follows from \eqref{ortho4}, \eqref{cts4}, Theorems \ref{bdsigma} and \ref{bdp} that
\begin{equation}\label{con2}
\begin{aligned}
e_{du,\ell+1}+\rho_{du}\eta_{du,\ell+1}&\leq\alpha_{4}(e_{du,\ell}+\rho_{du}\eta_{du,\ell})+\rho_{du}C_{du}(\Delta_{d\sigma,\ell}+\Delta_{p,\ell})\\
&\qquad+2\varepsilon_3^{-1}C^{du}_{\text{qo}}(e_{d\sigma,\ell+1}+e_{p,\ell+1}),\\
&\leq\alpha_{4}(e_{du,\ell}+\rho_{du}\eta_{du,\ell})\\
&\qquad+C(\eta_{\sigma,\ell}+\eta_{p,\ell}+\eta_{\sigma,\ell+1}+\eta_{p,\ell+1}),
\end{aligned}
\end{equation}
where $\rho_{du}=(1-\varepsilon_3)/C_{du}$, $\alpha_{4}=\frac{C^{du}_{\up}+\rho_{du}\beta_{du}}{C^{du}_{\up}+\rho_{du}}<1$, and $C^{du}_{\up}$ is the multiplicative constant in Theorem \ref{bddu}. Then using \eqref{con2}, \eqref{reductionsigma}, \eqref{reductionp}, and induction, we have  
\begin{equation}\label{reductiondu}
e_{du,\ell}+\rho_{du}\eta_{du,\ell}\leq C\gamma_{3}^{\ell},
\end{equation}
for $\gamma_{3}:=\sqrt{\max(\alpha_4,\gamma_1,\gamma_2)}$.
Combining \eqref{reductionsigma}, \eqref{reductionp}, and \eqref{reductiondu}, the proof is complete by taking $\gamma=\max(\gamma_1,\gamma_2,\gamma_3)$ and $\zeta=1+C_{\varepsilon_1}$.
\end{proof}

{In the proof of Theorem \ref{convergenceAMFEM1}, the contraction \eqref{reductionsigma} of $\|\sigma-\sigma_{\ell}\|_{H\Lambda^{k-1}}$ is guaranteed by the D\"orfler marking \eqref{mark1} and utilizing the quasi-error 
$\|\sigma-\sigma_h\|^2+\rho_\sigma\eta_{\sigma}^2(\Th)+(1+C_{\varepsilon_1})\|d(\sigma-\sigma_h)\|^2$ motivated by the orthogonality \eqref{ortho1} on $d\sigma$ and partial orthogonality \eqref{ortho2} on $\sigma$. Technically speaking, the negative term $-\Delta _{d\sigma,\ell}$ in \eqref{ortho1} is used to balance
the term involving $\Delta _{d\sigma,\ell}$ in \eqref{ortho2}. For $p$ and $du$, \eqref{ortho3} and \eqref{ortho4} seem not strong enough to produce contraction using a single marking step. Hence convergence of $\|p-p_{\ell}\|$ and $\|d(u-u_{\ell})\|$ is guaranteed by separate markings \eqref{mark1} and \eqref{mark2}}. The marking procedure at \textbf{Step 3} can be adjusted according to practical purpose. 

\begin{remark}\label{remarkAMFEM1}
If $\mathfrak{H}^{k}=\{0\}$, marking \eqref{mark2} disappears. In this case, the contraction in Theorem \ref{convergenceAMFEM1} reduces to
\begin{equation*}
\begin{aligned}
&\|\sigma-\sigma_{\ell}\|^2+\zeta\|d(\sigma-\sigma_\ell)\|^2+\|d(u-u_{\ell})\|^2\\
&\qquad+\rho_\sigma\eta^{2}_{\sigma}(\mathcal{T}_{\ell})+\rho_{du}\eta^{2}_{du}(\mathcal{T}_{\ell})\leq C_{0}\gamma^{\ell},
\end{aligned}
\end{equation*}
which is different from the contraction given in Theorem 20 in \cite{CW2017}. In particular, it guarantees $\|\sigma-\sigma_{\ell}\|_{H\Lambda^{k-1}}\rightarrow0$ while the AMFEM in \cite{CW2017} provides $\|d(\sigma-\sigma_{\ell})\|\rightarrow0$.
\end{remark}

\begin{remark}\label{kn2}
Consider the case $k=n$, \eqref{DHL} is just the traditional mixed method for solving the mixed formulation of Poisson's equation \eqref{compactHL} in $\mathbb{R}^{n}$ with $d^{n-1}=\emph{div}, \delta^n=-\nabla, \mathfrak{H}^n=\{0\}$ and $du=0, p=0$. In this case, markings \eqref{mark2} and \eqref{mark3} disappear, and thus \textsf{AMFEM1} conincides with \textsf{AMFEM2} given below. To prove the contraction and quasi-optimality of their AMFEM on a simply connected polygon in $\mathbb{R}^{2}$, the authors in \cite{CHX2009} constructed a technical quasi-orthogonality result
\begin{equation}\label{CHX}
(1-\varepsilon)e_{\sigma,\ell+1}\leq e_{\sigma,\ell}-\Delta_{\sigma,\ell}+\frac{C_{\mathcal{T}_0}}{\varepsilon}\osc_\ell,
\end{equation}
for any $0<\varepsilon<1$, $\osc_\ell:=\|h_{\ell}(f-f_{\mathcal{T}_{\ell}})\|^2$, where $h_{\ell}$ is the meshsize function with $h_{\ell}|_{K}=|K|^{\frac{1}{2}}$ for $K\in\mathcal{T}_{\ell}$. {The authors in \cite{HX2011,HuYu2018} further extended the analysis in \cite{CHX2009} to develop quasi-optimal AMFEMs for Poisson's equation in $\mathbb{R}^3.$} 

Comparing to \eqref{ortho2}, the quasi-orthogonality \eqref{CHX} is sharper because $\osc_\ell\ll\Delta_{d\sigma,\ell}=\|f_{\mathcal{T}_{\ell}}-f_{\mathcal{T}_{\ell+1}}\|^2.$ {However, for convergence analysis of \textsf{AMFEM1}, the elementary result \eqref{ortho2} is enough. It should be noted that the sharper quasi-orthogonality \eqref{CHX} yields better convergence and complexity results. For example, when using the Brezzi--Douglas--Marini element $\mathcal{P}_{r+1}\Lambda^{n-1}(\mathcal{T}_\ell)\times\mathcal{P}_r\Lambda^n(\mathcal{T}_\ell)$, the $H(\emph{div})$ error $\|\sigma-\sigma_\ell\|_{H\Lambda^{n-1}}=O(h^{r+1})$ is a lower order error comparing to the $L^2$ error  $\|\sigma-\sigma_\ell\|=O(h^{r+2})$. In this case, the AMFEMs in \cite{CHX2009,HX2011,HuYu2018} for controlling $\|\sigma-\sigma_\ell\|$ have better convergence rate than the convergence rate of $\|\sigma-\sigma_\ell\|_{H\Lambda^{n-1}}$ given in Theorem \ref{opt} provided $(\sigma,f)$ belong to suitable approximation class.}
\end{remark}

\section{Optimality}\label{optimality}
\begin{algorithm}\label{AMFEM2}
\textsf{AMFEM2}. Given an initial mesh $\mathcal{T}_{0}$, marking parameters $0<\theta<1$, and an error tolerance $\tol>0$. Set $\ell=0$.
\begin{itemize}
\item[]\textbf{Step 1} Solve \eqref{DHL} on $\mathcal{T}_{\ell}$ to obtain the finite element solution $(\sigma_{\ell},u_{\ell},p_{\ell})$. 
\item[]\textbf{Step 2}  Compute the error indicator $\eta_{\sigma}(K)$ on each element $K\in\mathcal{T}_{\ell}$ and $\eta_{\ell}=\eta_{\sigma}(\mathcal{T}_{\ell})$. If $\eta_\ell\leq\tol$, return $(\sigma_\ell,u_\ell,p_\ell)$ and $\mathcal{T}_\ell$; else go to \textbf{Step 3}.
\item[]\textbf{Step 3} Select a subset $\mathcal{M}_{\ell}$ of $\mathcal{T}_{\ell}$ such that 
\begin{align*}
\eta_{\sigma}(\mathcal{M}_{\ell})\geq\theta\eta_{\sigma}(\mathcal{T}_{\ell}).
\end{align*}
\item[]\textbf{Step 4} Refine each element in $\mathcal{M}_{\ell}$ and necessary neighboring elements by a mesh refinement algorithm preserving shape regularity to get a conforming mesh $\mathcal{T}_{\ell+1}$. Set $\ell=\ell+1$ and go to \textbf{Step 1}.
\end{itemize}
\end{algorithm}

For the purpose of proving optimality, it seems questionable to impose separate markings \eqref{mark2} and \eqref{mark3} in \textsf{AMFEM1}. See \cite{CKNS2008} for the drawback of separate marking for data oscillation. Hence we propose \textsf{AMFEM2} with a single marking step. As a result, we are not able to control $\|p-p_\ell\|$ and $\|d(u-u_\ell)\|$. However, quasi-optimality of \textsf{AMFEM2} will follow as a compensation. \textsf{AMFEM2} is a contraction by \eqref{reductionsigma} in the proof of Theorem \ref{convergenceAMFEM1}. 
\begin{theorem}[contraction of \textsf{AMFEM2}]\label{contraction} For  $f\in H^1\Lambda^{k}(\mathcal{T}_0)$ with $1\leq k\leq n-1$ or $f\in L^2\Lambda^n(\Omega)$, let $\{(\sigma_{\ell},u_{\ell},p_{\ell}),\mathcal{T}_{\ell}\}_{\ell\geq0}$ be a sequence of finite element solutions and meshes generated by Algorithm \textsf{AMFEM2}. 
Then there exist $\zeta$, $\rho>0$ and $\alpha\in(0,1)$, depending only on $\Omega$, $\mathcal{T}_{0}$ and $\theta$, such that
 \begin{equation*}
 \begin{aligned}
&\|\sigma-\sigma_{\ell+1}\|_{H\Lambda^{k-1}}^{2}+\zeta\|d(\sigma-\sigma_{\ell+1})\|^2+\rho\eta^{2}_{\sigma}(\mathcal{T}_{\ell+1})\\
&\quad\leq\alpha\left\{\|\sigma-\sigma_{\ell}\|_{H\Lambda^{k-1}}^{2}+\zeta\|d(\sigma-\sigma_{\ell})\|^2+\rho\eta^{2}_{\sigma}(\mathcal{T}_\ell)\right\}.
 \end{aligned}
 \end{equation*}
\end{theorem}

To prove the quasi-optimality of \textsf{AMFEM2}, we need a localized upper bound, which can be viewed as a discrete and local version of the global upper bound $\eta_{\sigma}(\Th)$ for  $\|\sigma-\sigma_{h}\|_{H\Lambda^{k-1}}$. Following \cite{Demlow2017}, we apply the locally defined $V$-bounded cochain projection $\tilde{\pi}_{h}$ given by Falk and Winther \cite{FW2014}. Let $\mathcal{T}_{H}$ be a refinement of $\Th$ and $\mathcal{R}_{H}=\mathcal{R}_{\mathcal{T}_{H}\rightarrow\mathcal{T}_{h}}$ be the set of refined elements in $\mathcal{T}_{H}$.  Given $\tilde{\pi}_{H}$ on $\mathcal{T}_{H}$ and $K\in\mathcal{T}_{H}$, there is a subdomain $D_{K}$  which contains $K$ and aligns with $\mathcal{T}_{H}$, such that $\#\{K^{\prime}\in\mathcal{T}_{H}: K^{\prime}\subset D_{K}\}\lesssim1$. In addition, $\tilde{\pi}_{H}$ is a local projection (see also (2.9) in \cite{Demlow2017}) in the sense that  
\begin{equation}\label{local}
v|_{D_{K}}\in V_{H}|_{D_{K}}\Rightarrow(v-\tilde{\pi}_{H}v)|_{K}=0.
\end{equation} 

To specify the dependence of $\eta_{\sigma}(K)$ on the mesh and finite element solution, we use the notation $\eta_{\sigma,\Th}(\sigma_{h},K)=\eta_{\sigma}(K)$ and
$\eta^2_{\sigma,\Th}(\sigma_{h},\mathcal{M})=\sum_{K\in\mathcal{M}}\eta^2_{\sigma,\Th}(\sigma_{h},K)$ for $\mathcal{M}\subset\mathcal{T}_{h}$ throughout the rest of this section.
\begin{theorem}[localized upper bound]\label{disupper}
Let $\Th$ be a conforming refinement of $\mathcal{T}_{H}$, $\sigma_{h}$ and $\sigma_{H}$ be the finite element solution approximating $\sigma$ on $\Th$ and $\mathcal{T}_{H}$, respectively. For $f\in H^1\Lambda^{k}(\mathcal{T}_H)$ with $1\leq k\leq n-1$ or $f\in L^2\Lambda^n(\Omega)$, there exists $\widetilde{\mathcal{R}}_{H}\subset\mathcal{T}_{H}$, which is the union of $\mathcal{R}_{H}$ and a collection of some neighboring elements of $\mathcal{R}_{H}$, $\#\widetilde{\mathcal{R}}_{H}\lesssim\#\mathcal{R}_{H}$, such that 
\begin{equation*}
\begin{aligned}
\|\sigma_{h}-\sigma_{H}\|_{H\Lambda^{k-1}}^{2}\leq C_{\loc}\eta^{2}_{\sigma,\mathcal{T}_{H}}(\sigma_{H},\widetilde{\mathcal{R}}_{H}),
\end{aligned}
\end{equation*}
where $C_{\loc}$ depends solely on $\Omega$ and the shape regularity of $\Th, \mathcal{T}_{H}$.
\end{theorem}
\begin{proof}
Consider $2\leq k\leq n$. We first focus on the proof of 
$\|\sigma_{h}-\sigma_{H}\|^{2}\lesssim\eta^{2}_{\sigma,\mathcal{T}_{H}}(\sigma_{H},\widetilde{\mathcal{R}}_{H}).$
Throughout the proof, $\widetilde{\mathcal{R}}_{H}$ always denotes the union of $\mathcal{R}_{H}$ and a collection of some neighboring elements of $\mathcal{R}_{H}$, $\#\widetilde{\mathcal{R}}_{H}\lesssim\#\mathcal{R}_{H}$. However, $\widetilde{\mathcal{R}}_{H}$ can vary from step to step and $\#\widetilde{\mathcal{R}}_{H}$ depends on the local property of $\tilde{\pi}_{h}$ and the coefficient-wise Scott--Zhang interpolation $\mathcal{I}_{h}$ \cite{SZ1990}. Similar to the proof of Theorem \ref{bdsigma}, we consider the discrete Hodge decomposition $\sigma_{h}-\sigma_{H}=dv_{h}^{1}+\delta_{h}v_{h}^{2}+q_{h}$, where $v_{h}^{1}\in\mathfrak{Z}_{h}^{\perp},\ v_{h}^{2}=dw_{h}\in\mathfrak{Z}_{h}^{*\perp}=\mathfrak{B}_{h},\ w_{h}\in\mathfrak{Z}^{\perp}_{h}$, and $q_{h}\in\mathfrak{H}_{h}.$ 

Let $v_{h}^{1}=d\varphi_1+z_1$ be a regular decomposition of $v_{h}^{1}$. In view of \eqref{local}, 
\begin{equation}\label{localsupp}
\supp(\tilde{\pi}_{h}z_1-\tilde{\pi}_{H}\tilde{\pi}_{h}z_1)\subseteq\bigcup_{K\in\widetilde{\mathcal{R}}_{H}} K. \end{equation}
Since $\tilde{\pi}_{h}$ is a cochain projection, $dv_{h}^{1}=dz_1=d\tilde{\pi}_{h}z_1$. Then by $\sigma_{h}\in\mathfrak{Z}_{h}^{\perp}$, $\sigma_{H}\in\mathfrak{Z}_{H}^{\perp}$, element-wise integration by parts, and \eqref{localsupp}, we have
\begin{equation}\label{disupper1}
\begin{aligned}
        \|dv_{h}^{1}\|^{2}&=\ab{\sigma_{h}-\sigma_{H}}{dv_{h}^{1}}\\
        &=\ab{-\sigma_{H}}{d(\tilde{\pi}_{h}z_1-\tilde{\pi}_{H}\tilde{\pi}_{h}z_1)}\\
        &=\sum_{K\in\widetilde{\mathcal{R}}_{H}}-\ab{\delta\sigma_{H}}{\tilde{\pi}_{h}z_1-\tilde{\pi}_{H}\tilde{\pi}_{h}z_1}_{K}-\int_{\partial K}\tr\star\sigma_{H}\wedge\tr(\tilde{\pi}_{h}z_1-\tilde{\pi}_{H}\tilde{\pi}_{h}z_1)\\
        &\lesssim\sum_{K\in\widetilde{\mathcal{R}}_{H}}\big(h_{K}^{2}\|\delta\sigma_{H}\|_{K}^{2}+h_{K}\|\lr{\tr\star\sigma_{H}}\|_{\partial K}^{2}\big)^{\frac{1}{2}}(I_{K}+J_{K}),
\end{aligned}
\end{equation}
where 
$$I_{K}=h_{K}^{-1}\|\tilde{\pi}_{h}z_1-\tilde{\pi}_{H}\tilde{\pi}_{h}z_1\|_{K},\quad J_{K}=h_{K}^{-\frac{1}{2}}\|\tr(\tilde{\pi}_{h}z_1-\tilde{\pi}_{H}\tilde{\pi}_{h}z_1)\|_{\partial K}.$$
Splitting $\tilde{\pi}_{h}z_1-\tilde{\pi}_{H}\tilde{\pi}_{h}z_1=(\tilde{\pi}_{h}z_1-\mathcal{I}_{h}z_1)+(\mathcal{I}_{h}z_1-\tilde{\pi}_{H}\mathcal{I}_{h}z_1)+\tilde{\pi}_{H}(\mathcal{I}_{h}z_1-\tilde{\pi}_{h}z_1),$ Demlow proved that (see the proof Lemma 3 in \cite{Demlow2017}) 
\begin{equation}\label{Demlowapprox}
I_{K}+J_{K}\lesssim|z_1|_{H^{1}(D_{K})}.
\end{equation}
Combining \eqref{disupper1}, \eqref{Demlowapprox} with the Cauchy--Schwarz inequality and $|z_1|_{H^{1}}\lesssim\|v_{h}^{1}\|_{H\Lambda^{k-2}}$$\lesssim\|dv_{h}^{1}\|$, we have
\begin{equation}\label{disupperdvh1}
\|dv_{h}^{1}\|^{2}\lesssim\sum_{K\in\widetilde{\mathcal{R}}_{H}}h_{K}^{2}\|\delta\sigma_{H}\|_{K}^{2}+h_{K}\|\lr{\tr\star\sigma_{H}}\|_{\partial K}^{2}.
\end{equation}

Let $w_{h}=d\varphi_{2}+z_{2}$ be a regular decomposition of $w_{h}$. It follows from \eqref{DHL} and $v_{h}^{2}=dw_{h}=dz_{2}=d\tilde{\pi}_{h}z_{2}$ that
\begin{equation*}
\begin{aligned}
\|\delta_{h}v_{h}^{2}\|^{2}&=\ab{d(\sigma_{h}-\sigma_{H})}{v_{h}^{2}}\\
        &=\ab{d(\sigma_{h}-\sigma_{H})}{d(\tilde{\pi}_{h}z_{2}-\tilde{\pi}_{H}\tilde{\pi}_{h}z_{2})}\\
        &=\ab{f-d\sigma_{H}}{d(\tilde{\pi}_{h}z_{2}-\tilde{\pi}_{H}\tilde{\pi}_{h}z_{2})}.
\end{aligned}
\end{equation*}
Then by similar argument in proving \eqref{disupperdvh1} and a series of bounds 
$$\|z_{2}\|_{H^{1}}\lesssim\|w_{h}\|_{H\Lambda^{k-1}}\lesssim\|dw_{h}\|=\|v_{h}^{2}\|\lesssim\|\delta_{h}v_{h}^{2}\|,$$ we have
\begin{equation}\label{disupperdeltavh2}
\begin{aligned}
        \|\delta_{h}v_{h}^{2}\|^{2}&\lesssim\sum_{K\in\widetilde{\mathcal{R}}_{H}}\big(h_{K}^{2}\|\delta(f-d\sigma_{H})\|_{K}^{2}+h_{K}\|\lr{\tr\star(f-d\sigma_{H}}]\|_{\partial K}^{2}\big)^{\frac{1}{2}}\\
        &\times\big(h_{K}^{-1}\|\tilde{\pi}_{h}z_{2}-\tilde{\pi}_{H}\tilde{\pi}_{h}z_{2}\|_{K}^{2}+h_{K}^{-\frac{1}{2}}\|\tr(\tilde{\pi}_{h}z_{2}-\tilde{\pi}_{H}\tilde{\pi}_{h}z_{2})\|_{\partial K}\big)\\
        &\lesssim\sum_{K\in\widetilde{\mathcal{R}}_{H}}h_{K}^{2}\|\delta(f-d\sigma_{H})\|_{K}^{2}+h_{K}\|\lr{\tr\star(f-d\sigma_{H})}\|_{\partial K}^{2}.
\end{aligned}
\end{equation}

In the end, the harmonic component is controlled by
\begin{equation}\label{disupperqh}
\begin{aligned}
\|q_{h}\|&=\ab{\sigma_{h}-\sigma_{H}}{\frac{q_{h}}{\|q_{h}\|}}\\
        &=\ab{\sigma_{h}-\sigma_{H}}{\frac{q_{h}}{\|q_{h}\|}-P_{\mathfrak{H}_{H}}\frac{q_{h}}{\|q_{h}\|}}\\
        &\leq\delta(\mathfrak{H}_{h},\mathfrak{H}_{H})\|\sigma_{h}-\sigma_{H}\|,
\end{aligned}
\end{equation}
where $\delta(\mathfrak{H}_{h},\mathfrak{H}_{H})\leq C(\mathcal{T}_{0},\Omega)<1$ by Lemma \ref{gaph}.

Combining the above three bounds \eqref{disupperdvh1}, \eqref{disupperdeltavh2}, and \eqref{disupperqh}, we prove that
\begin{equation}\label{locdisbdsigma}
\begin{aligned}
 \|\sigma_{h}-\sigma_{H}\|^{2}&=\|dv_{h}^{1}\|^{2}+\|\delta_{h}v_{h}^{2}\|^{2}+\|q_{h}\|^{2}\\
        &\leq\frac{1}{1-C(\mathcal{T}_{0},\Omega)^{2}}\big(\|dv_{h}^{1}\|^{2}+\|\delta_{h}v_{h}^{2}\|^{2}\big)\\
        &\lesssim\sum_{K\in\widetilde{\mathcal{R}}_{H}}\eta^{2}_{\sigma,\mathcal{T}_{H}}(\sigma_{H},K).
\end{aligned}
\end{equation}
When $k=1$, the proof is the same but without the boundary component $dv_{h}^{1}$. 

When $1\leq k\leq n-1$, let $\sigma_{h}-\sigma_{H}=d\varphi_{3}+z_{3}$ be a regular decomposition of $\sigma_{h}-\sigma_{H}.$ Then $d(\sigma_{h}-\sigma_{H})=dz_{3}=d\tilde{\pi}_{h}z_{3}$ and
\begin{align*}
\|d(\sigma_{h}-\sigma_{H})\|^{2}&=\ab{d(\sigma_{h}-\sigma_{H})}{d(\tilde{\pi}_{h}z_{3}-\tilde{\pi}_{H}\tilde{\pi}_{h}z_{3})}\\
&=\ab{f-d\sigma_{H}}{d(\tilde{\pi}_{h}z_{3}-\tilde{\pi}_{H}\tilde{\pi}_{h}z_{3})}.
\end{align*}
Following the proof of \eqref{disupperdeltavh2}, we obtain
\begin{equation}\label{locdisbddsigma}
\|d(\sigma_{h}-\sigma_{H})\|^{2}\lesssim\sum_{K\in\widetilde{\mathcal{R}}_{H}}h_{K}^{2}\|\delta(f-d\sigma_{H})\|_{K}^{2}+h_{K}\|\lr{\tr\star(f-d\sigma_{H})}\|_{\partial K}^{2}.
\end{equation}
When $k=n$, \begin{equation}\label{knf}
\|d(\sigma_h-\sigma_H)\|^2=\|f_{\Th}-f_{\mathcal{T}_H}\|^2=\sum_{K\in\mathcal{R}_H}\|f_{\Th}-f_{\mathcal{T}_H}\|_K^2.
\end{equation}
Combining \eqref{locdisbdsigma} ,\eqref{locdisbddsigma}, and \eqref{knf}, the proof is complete.
\end{proof}

In addition, we need several ingredients as in \cite{CKNS2008}. 
For a subset $\mathcal{M}\subset\Th$, define $\osc^2_{\Th}(\tau_{h},f,\mathcal{M}):=\sum_{K\in\mathcal{M}}\osc^2_{\mathcal{T}_{h}}(\tau_{h},f,K).$ Recall that $V_{h}^{k-1}$ and $V_{H}^{k-1}$ are finite element subspaces of $H\Lambda^{k-1}(\Omega)$ based on $\Th$ and $\mathcal{T}_{H}$, respectively. The first part of the  following lemma follows from the same proof of Corollary 3.5 in \cite{CKNS2008}. The second part is straightforward by the definition of $\osc_{\Th}$.
\begin{lemma}[properties of oscillation]\label{perturb}
Let $\mathcal{T}_{h}$ be a conforming refinement of $\mathcal{T}_{H}$. For any $\tau_{h}\in V_{h}^{k-1}$ and $\tau_{H}\in V_{H}^{k-1}$, we have
\begin{equation*}
\osc^2_{\mathcal{T}_{H}}(\tau_{H},f,\mathcal{T}_{H}\cap\mathcal{T}_{h})\leq2\osc^2_{\mathcal{T}_{h}}(\tau_{h},f,\mathcal{T}_{H}\cap\mathcal{T}_{h})+C_{\osc}\|d(\tau_{h}-\tau_{H})\|^{2}.
\end{equation*}
where $C_{\osc}$ depends only on $\mathcal{T}_{0}$. In addition, the dominance holds
$$\osc_{\Th}(\tau_{h},f,K)\leq\eta_{\sigma,\Th}(\tau_{h},K),\quad K\in\Th.$$
\end{lemma}

Before presenting the optimality result, we make several assumptions.
\begin{assumption}\label{assump}
 We assume the following properties of Algorithm \textsf{AMFEM2}.
\begin{itemize}
\item[](a) The marking parameter $\theta$ satisfies $\theta\in(0,\theta_{*})$, where
$$\theta_{*}^{2}=\frac{C_{\low}}{1+(C_{\osc}+1)C_{\loc}}.$$
\item[](b) \textbf{Step 3} selects a set $\mathcal{M}_{\ell}$ with minimal cardinality.
\item[](c) \textbf{Step 4} generates a sequence of meshes $\{\mathcal{M}_{\ell}\}_{\ell\geq0}$ satisfying the cardinality estimate 
$$\#\mathcal{T}_{\ell}-\#\mathcal{T}_{0}\lesssim\sum_{i=0}^{\ell-1}\#\mathcal{M}_{i}.$$
\end{itemize}
\end{assumption}
Assumption (b) can be guaranteed by sorting $\{\eta_{\sigma,\mathcal{T}_\ell}(\sigma_\ell,K)\}_{K\in\mathcal{T}_\ell}$. Assumption (c) holds provided the newest vertex bisection when $n=2$ \cite{Mitchell1990} or its generalization when $n\geq3$ \cite{Traxler1997} is applied in \textbf{Step 4} and a matching condition holds the initial mesh $\mathcal{T}_0$, see \cite{Stevenson2008,BDD2004}.

Let $\mathbb{T}$ denote the collection of conforming refinements of $\mathcal{T}_0$ produced by the newest vertex bisection or its higher dimensional generalization. Let $\mathbb{T}_N=\{\mathcal{T}_{h}\in\mathbb{T}: \#\mathcal{T}_{h}-\#\mathcal{T}_{0}\leq N\}$. For $s>0$, define the approximation class
$$\mathbb{A}_{s}=\{(\tau,g)\in H\Lambda^{k-1}(\Omega)\times L^{2}\Lambda^{k}(\Omega): |(\tau,g)|_{s}=\sup_{N>0}\left\{N^{s}E(N;\tau,g)\right\}<\infty\},$$ 
where
$$E(N;\tau,g):=\inf_{\mathcal{T}_{h}\in\mathbb{T}_{N}}\inf_{\tau_{h}\in V_{h}^{k-1}}\left\{\|\tau-\tau_{h}\|_{H\Lambda^{k-1}}^{2}+\osc_{\Th}^{2}(\tau_{h},g,\mathcal{T}_{h})\right\}^{\frac{1}{2}}.$$
Combining the lower bound in Theorem \ref{eff}, the localized upper bound in Theorem \ref{disupper}, Lemma \ref{perturb}, and the contraction in Theorem \ref{contraction}, the proof of the optimality of \textsf{AMFEM2} is almost the same as the optimality proof in \cite{CKNS2008}, see Lemmas 5.9, 5.10 and Theorem 5.11 in \cite{CKNS2008} for details.
\begin{theorem}[quasi-optimality of \textsf{AMFEM2}]\label{opt}
Let Assumption \ref{assump} be satisfied. For $f\in H^1\Lambda^{k}(\mathcal{T}_0)$ with $1\leq k\leq n-1$ or $f\in L^2\Lambda^n(\Omega)$, let $\{(\sigma_{\ell},u_{\ell},p_{\ell}),\mathcal{T}_{\ell}\}_{\ell\geq0}$ be a sequence of finite element solutions and meshes generated by Algorithm \textsf{AMFEM2}. There exists a constant $C_{\opt}$ depending only on $\Omega$, $\mathcal{T}_0$, $\theta, \theta_*$, such that
\begin{equation*}
\left\{\|\sigma-\sigma_{\ell}\|_{H\Lambda^{k-1}}^{2}+\osc^{2}_{\mathcal{T}_{\ell}}(\sigma_{\ell},f,\mathcal{T}_{\ell})\right\}^{\frac{1}{2}}\leq C_{\opt}|(\sigma,f)|_{s}(\#\mathcal{T}_{\ell}-\#\mathcal{T}_{0})^{-s}.
\end{equation*}
\end{theorem}
    
\section{Concluding remarks}\label{conclusion} 
In this paper, we have developed two adaptive mixed finite element methods \textsf{AMFEM1} and \textsf{AMFEM2} for solving the Hodge Laplacian problem on bounded Lipschitz domains. \textsf{AMFEM1} is convergent w.r.t.~$\|\sigma-\sigma_h\|_{H\Lambda^{k-1}}^2+\|p-p_h\|+\|d(u-u_h)\|^2$ while \textsf{AMFEM2} is quasi-optimal w.r.t.~$\|\sigma-\sigma_h\|_{H\Lambda^{k-1}}$. Our results are presented in the framework of FEEC. For translation of results from FEEC into $H(\text{curl})$ and $H(\text{div})$ in $\mathbb{R}^{3}$, readers are referred to \cite{AFW2006,AFW2010,DH2014,CW2017}.
    
Although the a posteriori upper bound $\eta_{\text{DH}}(\Th)$ for
$NE_h=(\|\sigma-\sigma_{h}\|_{H\Lambda^{k-1}}^{2}+\|u-u_{h}\|^{2}_{H\Lambda^{k}}+\|p-p_{h}\|^{2})^{\frac{1}{2}}$ 
is available, we are not able to develop a convergent AMFEM for controlling the error $NE_h$ in the natural norm. The main difficulty comes from the term $\mu\|u_{h}\|$ in $\eta_{\text{DH}}$, which is an obstacle against proving efficiency and estimator reduction in Lemma \ref{lemma2}. However, in the absence of harmonic forms, the Arnold--Falk--Winther method \eqref{DHL} for the Hodge Laplacian reduces to a conforming method and $\eta_{\text{DH}}|_K=\big(\eta_{0}^2(K)+\eta_{-1}^2(K)+\eta_{\mathfrak{H}}^2(K,p_h)\big)^{\frac{1}{2}}$ is locally efficient up to data oscillation. In this case, the convergence of AMFEM based on $\eta_{\text{DH}}$ follows from the plain convergence result in \cite{MSV2008}.
    
\section*{Acknowledgments}
The author would like to thank Professor Alan Demlow for his remark on the fineness assumption about the initial mesh made in an earlier version of this paper. The author would also like to thank the referee for the comments that improve the presentation of this paper.

\bibliographystyle{siamplain}

\end{document}